\documentclass[10pt]{amsart}
\usepackage{amsmath}
\usepackage{amssymb}
\usepackage{amsfonts}
\usepackage{amsthm}
\usepackage{enumerate}
\usepackage[all]{xy}
\usepackage[latin1]{inputenc}
\usepackage{mathdots}
\usepackage{graphicx}
\usepackage{latexsym}
\usepackage{mathrsfs}
\PassOptionsToPackage{usenames,dvipsnames,svgnames}{xcolor}
\usepackage{tikz}

\usetikzlibrary{shapes,arrows,positioning,automata}
\definecolor{cof}{RGB}{219,144,71}
\definecolor{pur}{RGB}{186,146,162}
\definecolor{greeo}{RGB}{91,173,69}
\definecolor{greet}{RGB}{52,111,72}

\input xy

\newtheorem{Theo}{Theorem}[section]
\newtheorem{Prop}[Theo]{Proposition}
\newtheorem{Cor}[Theo]{Corollary}
\newtheorem{Lemma}[Theo]{Lemma}

\theoremstyle{definition}

\newtheorem{Exam}[Theo]{Example}

\newcommand{\mmid}{\mid}
\newcommand{\rest}{|}
\def\mystrut(#1,#2){\vrule height #1pt depth #2pt width 0pt}
\newcommand{\rep}{{\rm rep}}

\newcommand{\Hom}{{\rm Hom}}
\newcommand{\Ext}{{\rm Ext}}

\newcommand{\Z}{\mathbb{Z}}
\newcommand{\C}{\mathcal{C}}

\newcommand{\p}{{}^\perp}%{\hspace{-2pt}}}
\newcommand{\ind}{\operatorname{ind}}
\newcommand{\hyp}{\mathcal J}
\newcommand{\lat}{\mathcal L}
\newcommand{\reg}{{\mathcal R}{\rm eg}}
\newcommand{\add}{\operatorname{add}}

\PassOptionsToPackage{usenames,dvipsnames}{color}
\begin{document}

\title[Semi-stable subcategories]{Semi-stable subcategories for Euclidean quivers}

\dedicatory{Dedicated to the memory of Dieter Happel}

\author{Colin Ingalls}
\address{Colin Ingalls, Dept. of Math. and Stat, University of New Brunswick, Fredericton, NB, Canada, E3B 5A3}
\author{Charles Paquette}
\address{Charles Paquette, Dept. of Math., University of Connecticut, Storrs, CT, USA, 06269-3009}
\author{Hugh Thomas}
\address{Hugh Thomas, Dept. of Math. and Stat, University of New Brunswick, Fredericton, NB, Canada, E3B 5A3}


\maketitle

\begin{abstract} In this paper, we study the semi-stable subcategories of the category of representations of a Euclidean quiver, and the possible intersections of these subcategories. Contrary to the Dynkin case, we find out that the intersection of semi-stable subcategories may not be semi-stable.  However, only a finite number of exceptions occur, and we give a description of these subcategories. Moreover, one can attach a simplicial fan in $\mathbb{Q}^n$ to any acyclic quiver $Q$, and this simplicial fan allows one to completely determine the canonical presentation of any element in $\mathbb{Z}^n$. This fan has a nice description in the Dynkin and Euclidean cases: it is described using an arrangement of convex codimension-one subsets of $\mathbb{Q}^n$, each such subset being indexed by a real Schur root or a set of quasi-simple objects.  This fan also characterizes
when two different stability conditions give rise to the same semi-stable
subcategory.
\end{abstract}

\section{Introduction}

\subsection*{Semi-stable subcategories}
In Mumford's geometric invariant theory, in order to form the quotient of a variety
by a group action, one first replaces the variety by its semi-stable points.
This important idea was interpreted in the setting of quiver representations
by King \cite{King}.

Let $Q$ be a (possibly disconnected) acyclic quiver, and $k$ an algebraically closed field. In this paper, all quivers are acyclic. We write $\rep(Q)$ for the category
of finite-dimensional representations of $Q$ over $k$.
For $\theta,$ a $\mathbb Z$-linear
functional on the Grothendieck group of $\rep(Q)$, referred to as a
stability condition,
King shows how to define the
subcategory of $\theta$-semi-stable objects in $\rep(Q)$.  It is
immediate from King's definition, which we shall recall below, that
for any $\theta$, the
$\theta$-semi-stable subcategory of $\rep(Q)$ is abelian and
extension-closed.
It is therefore natural to ask which abelian and extension-closed subcategories
can arise as semi-stable subcategories.  (All subcategories considered in this paper are full subcategories.  Also, when we refer to abelian subcategories,
we mean exact abelian subcategories, that is to say, subcategories that
are abelian with respect to the abelian structure of the ambient category.)

In \cite{IT}, two of the authors
of the present paper showed that when $Q$ is Dynkin, \emph{any}
abelian and extension-closed subcategory arises as the
$\theta$-semi-stable subcategory of a suitable choice of $\theta$.
In fact, somewhat more is known; see \cite[Theorem 1.1]{IT}.
For any acyclic quiver $Q$, the following conditions on an abelian,
extension-closed subcategory $\mathcal A$ of $\rep(Q)$ are equivalent; see Proposition \ref{equiv_thick}.
\begin{itemize}
\item $\mathcal A$ admits a projective generator.
\item $\mathcal A$ is generated by the elements of an exceptional sequence.
\item $\mathcal A$ is equivalent to the category of representations of
some quiver $Q'$.
\end{itemize}
We refer to subcategories satisfying these equivalent
conditions as finitely generated.
In \cite{IT}, it was shown for any acyclic quiver $Q,$
that any finitely generated, abelian and extension-closed subcategory of $\rep (Q)$
arises as the semi-stable subcategory for
some stability condition.  This resolves the Dynkin case because in that
case, every abelian and extension-closed subcategory is
finitely generated.

The next case one would hope to settle is the Euclidean case. In (ii) below, we refer to possibly disconnected Euclidean quivers.
A \emph{possibly disconnected Euclidean quiver} is defined to be a quiver with one Euclidean
component and all other components (if any) Dynkin. A \emph{regular object} is then a representation which is a direct sum of regular representations of the Euclidean component and/or representations of Dynkin components. We resolve the Euclidean case as follows:

{\renewcommand{\theTheo}{\ref{one}}
\begin{Theo} For $Q$ a Euclidean quiver, an abelian and extension-closed
subcategory $\mathcal B$ of $\rep (Q)$ is the subcategory of
$\theta$-semi-stable representations for some $\theta$ if and only if either:
\begin{enumerate}[$(i)$]
\item $\mathcal B$ is finitely generated, or
\item there exists some abelian, extension-closed,
finitely generated subcategory $\mathcal A$ of
$\rep(Q)$, equivalent to the representations of a Euclidean quiver
(possibly disconnected),
and $\mathcal B$ consists of all the regular objects of $\mathcal A$.
\end{enumerate}
\end{Theo}
\addtocounter{Theo}{-1}}

The semi-stable subcategories of the second type can also be described
more explicitly as follows:

{\renewcommand{\theTheo}{\ref{oneone}}
\begin{Prop} The semi-stable subcategories in (ii) of Theorem \ref{one}
can also be described as those abelian, extension-closed
subcategories of the regular part of $\rep(Q)$, which
contain infinitely many indecomposable objects from each tube.
\end{Prop}
\addtocounter{Theo}{-1}}

From \cite[page 126]{Len}, the category of coherent sheaves over a weighted projective line, in the tame domestic case, is derived-equivalent to $\rep(Q)$ for $Q$ Euclidean.  Hence, we get the following, where two subcategories $\mathcal B_1,\mathcal B_2$ of a given hereditary abelian category $\mathcal{H}$ are \emph{orthogonal} if for any objects $M_1 \in \mathcal B_1, M_2 \in \mathcal B_2$, $\Hom_{\mathcal{H}}(M_1, M_2) = 0 = \Hom_{\mathcal{H}}(M_2, M_1)$ and $\Ext^1_{\mathcal{H}}(M_1, M_2) = 0 = \Ext^1_{\mathcal{H}}(M_2, M_1)$.

{\renewcommand{\theTheo}{\ref{coh}}
\begin{Theo}
Let $\mathbb{X}$ be a weighted projective line of tame domestic type.  Then an abelian and extension-closed
subcategory $\mathcal B$ of coh$(\mathbb{X})$ is the subcategory of
$\theta$-semi-stable objects for some $\theta$ if and only if either:
\begin{enumerate}[$(i)$]
\item $\mathcal{B}$ is equivalent to the category of coherent sheaves over a weighted projective line,
\item $\mathcal{B}$ is equivalent to the torsion part of a category of coherent sheaves over a weighted projective line,
\item $\mathcal B$ is the additive hull of two abelian, extension-closed subcategories $\mathcal B_1,\mathcal B_2$ that are orthogonal
to each other and where $\mathcal{B}_1$ is as in (i) or (ii) and $\mathcal{B}_2$ is of finite representation type.
\end{enumerate}
\end{Theo}
\addtocounter{Theo}{-1}}

\subsection*{Equivalence of stability conditions, canonical decomposition
of dimension vectors}
There is a great deal of interesting convex geometry associated to
stability conditions.  We can fix a dimension
vector $\alpha$, and ask for what stability conditions $\sigma$ there will be
objects of dimension vector $\alpha$ which are $\sigma$-semi-stable.  This
is shown to be a rational polyhedral cone \cite{DW3}, and these cones are
studied further in \cite{DW1}.  It is also of interest to consider the union
of the cones that arise in this way; this has been done in \cite{IOTW, Ch},
where the emphasis was on the general theory and on the Dynkin case.
We take a somewhat different approach: we draw similar pictures,
but interpret them differently, and we focus on the Dynkin and Euclidean cases.
They exhibit many
interesting features that are absent from the wild setting.
We consider all dimension vectors at once, focussing on the
subcategory of semi-stable objects which each stability condition
induces.

There is a natural equivalence relation on stability conditions, which
we call ss-equivalence,
where
$\theta$ and $\theta'$ are equivalent if they induce the same
semi-stable subcategory.
The set of stability conditions is naturally a free abelian group
of finite rank, dual to the Grothendieck group,
but in order to think about this equivalence,
it is easier to work in the corresponding
finite-dimensional vector space over $\mathbb Q$.
It is also convenient to use the Euler form on the Grothendieck group to
identify the Grothendieck group and its dual: we say
that $d_1$ and $d_2$ are ss-equivalent if $\langle d_1,-\rangle$ and
$\langle d_2, -\rangle$ are.
(We recall the
definition of the Euler form in Section \ref{sectiontwo}.)

Let $n$ denote the number of vertices of $Q$. There is a collection $\hyp$ of convex
codimension-one subsets
in $\mathbb Q^n$ such that
$d_1$ and $d_2$ are ss-equivalent if and only if they  lie in the same
subsets of $\hyp$.
Given an element $d\in \mathbb \Z^n$,
define
$$\hyp_d=\{J\in \hyp \mmid d\in J\}.$$ Then we have:

{\renewcommand{\theTheo}{\ref{ssequiv}}
\begin{Theo}
 Let $Q$ be a Dynkin or Euclidean quiver and $d_1,d_2 \in \Z^n$.
Then $d_1$ and $d_2$ are ss-equivalent if and only if $\hyp_{d_1}=\hyp_{d_2}$.
\end{Theo}
\addtocounter{Theo}{-1}}

Somewhat surprisingly, the geometry of $\hyp$, which controls ss-equivalence,
can also be used to describe canonical decompositions, and more generally, canonical presentations in the sense of \cite{IOTW}.
Given a dimension vector for $Q$, it was shown by Kac \cite{Kac}
that
the dimension vectors of the indecomposable summands of a generic
representation of dimension vector $d$ are well-defined.  The expression
of $d$ as the sum of the dimension vectors of the indecomposable summands
of a generic representation is called the canonical decomposition. In \cite{IOTW}, the authors have extended this canonical decomposition to any element in $\Z^n$. If $d$ is a dimension vector, then the canonical presentation of $d$ coincides with the canonical decomposition of $d$. In general, one can write $d = d_+ + d_-$, where $d_+$ is a dimension vector and $-d_-$ is the dimension vector of a projective representation whose top has a support disjoint from $d_+$. The canonical presentation of $d_-$ is just the decomposition of $d_-$ as a (negative) linear combination of the dimension vectors of the indecomposable projective representations. Then the canonical presentation of $d$ is just the sum of the canonical decomposition of $d_+$ and the canonical presentation of $d_-$.
If the same
summands appear in the canonical presentation of two vectors
(possibly with different multiplicities),
we say that the vectors are \emph{cp-equivalent}.
We show that cp-equivalence can also be characterized in terms of $\hyp$.

Let $\lat$ be the set of all intersections of subsets of
$\hyp$.  We order $\lat$ by inclusion.
This is a lattice, which we call the intersection lattice
associated to $\hyp$. For $L\in \lat$, define the faces of $L$ to be
the connected components of the set of points which are in $L$ but
not in any smaller intersection.  Define the faces of $\lat$ to be
the collection of all faces of all the elements of $\lat$.  Then we have
the following theorem:

{\renewcommand{\theTheo}{\ref{two}}
\begin{Theo}
Let $Q$ be a Dynkin or Euclidean quiver.  Two vectors $d_1, d_2$ are cp-equivalent if and only if
they lie in the same face of $\lat$.  \end{Theo}
\addtocounter{Theo}{-1}}

Note that if two vectors $d_1, d_2$
lie in the same face of $\lat$, then,
in particular, $\hyp_{d_1}=\hyp_{d_2}$.
It follows that if two dimension vectors are
cp-equivalent, then they are ss-equivalent.  This fact can also be
established directly.
The converse does not hold:
if $\hyp_{d_1}=\hyp_{d_2}$, then the smallest element of $\lat$ containing
$d_1$ is also the smallest element of $\lat$ containing $d_2$,
but since elements of $\lat$ are typically subdivided into more than one
face, it does not follow that $d_1$ and $d_2$ lie in the same face, so
they need not be cp-equivalent. Since cp-equivalence refines ss-equivalence, an
ss-equivalence class is a union of faces of $\lat$.

\subsection*{Posets of subcategories}
In Dynkin type, the semi-stable subcategories (or equivalently the abelian and extension-closed subcategories) form a lattice.  It
was shown in \cite{IT} that this poset is isomorphic to the lattice of noncrossing partitions
associated to the Weyl group corresponding  to $Q$.  These lattices
had already been studied by combinatorialists and group theorists and,
especially relevant for our purposes, they
play a central role in the construction of the dual Garside structure of
Bessis \cite{Be} on the corresponding Artin group, which also leads to
their use in constructing Eilenberg-Mac Lane spaces for the Artin groups
\cite{Br, BW}.
For these latter two uses,
the lattice property of this partial order is essential.

For this reason,
another motivation for our paper is to construct a potential replacement
lattice
in Euclidean type.  The most obvious choice for a lattice associated to a
general $Q$
would be to take
all the abelian and extension-closed subcategories.
However, this lattice
is very big and somehow non-combinatorial; already in Euclidean type,
it contains the Boolean lattice on the $\mathbb P^1(k)$-many tubes.
On the other hand, we could consider finitely generated abelian
and extension-closed subcategories.  It was shown in \cite{IT}
(Euclidean type) and \cite{IS} (general acyclic $Q$) that this yields the natural
generalization of the noncrossing partitions of the associated Weyl group.
However, for $Q$ non-Dynkin, this poset is typically
not a lattice \cite{Digne}, which makes it unsuitable.

Therefore, in pursuit of a suitable lattice of subcategories associated to $Q$
it seems that we need to consider a class of subcategories which
are not all finitely generated,
but not so broad as to include the full plethora of abelian, extension-closed subcategories.  The lattice property can be guaranteed in a
natural way if we can verify that our class of subcategories is closed under
intersections.  It turns out that the semi-stable subcategories do not
form a lattice, but it is quite easy to describe the subcategories that arise
as intersections of semi-stable subcategories; this class of subcategories
then (automatically) forms a lattice.  Specifically, we show:

{\renewcommand{\theTheo}{\ref{inters} {\bf (simplified form)}}
\begin{Theo}  Let $Q$ be a Euclidean quiver.
There are finitely many  subcategories of $\rep(Q)$
which arise as an intersection of semi-stable
subcategories, and which are not themselves semi-stable subcategories
for any stability condition. These subcategories are contained entirely
in the regular representations of $Q$.  Any such subcategory can be
written as the intersection of at most two semi-stable subcategories.
\end{Theo}
\addtocounter{Theo}{-1}}

In fact, we give an explicit description of these subcategories, which
we defer to the main body of the paper.

We hope to investigate the applicability of the lattice of these
subcategories to the problem
of constructing a dual Garside structure for Euclidean type Artin groups
in subsequent work.

\section{Some representation theory}\label{sectiontwo}

Let $Q = (Q_0,Q_1)$ be an (acyclic) quiver with $n$ vertices and let $k$ be an algebraically closed field.  For simplicity, we
shall assume that $Q_0=\{1,2,\ldots,n\}$. Our main concern is to study the semi-stable subcategories of $\rep (Q)$ and their intersections, when $Q$ is a \emph{Euclidean quiver}.  Unless otherwise specified, this includes the assumption
that $Q$ is connected.
When $Q$ is a Euclidean quiver, $\rep(Q)$ is of tame representation type and the representation theory of $Q$ is well understood. For the basic results concerning the structure of $\rep(Q)$, the reader
is referred to \cite{ASS, AS2}. We shall use many representation-theoretic results for $\rep(Q)$ and in particular, the structure of its Auslander-Reiten quiver.

Let $\langle \;,\; \rangle$ stand for the bilinear form defined on $\mathbb{Z}^n$ as follows.
If $d=(d_1,\ldots,d_n) \in \mathbb{Z}^n$ and $e=(e_1,\ldots,e_n) \in \mathbb{Z}^n$, then
$$\langle d,e \rangle = \sum_{i=1}^n{d_ie_i} \; - \hspace{-5pt} \sum_{\alpha:i \to j \in Q_1} d_{i}e_{j}.$$
This is known as the \emph{Euler form} associated to $Q$. This is a (non-symmetric) bilinear form defined on the Grothendieck group
$K_0(\rep(Q)) = \mathbb{Z}^n$ of $\rep(Q)$. For a representation $M$ in $\rep(Q)$, we denote by $d_M$ its dimension vector (which we identify with its
class in the Grothendieck group). A crucial property of the above bilinear form is the following; see for example \cite[Prop. III 3.13]{ASS}.
\begin{Prop} Let $M,N \in \rep(Q)$.  Then
$$\langle d_M,d_N \rangle = {\rm dim}_k \Hom(M,N) - {\rm dim}_k\Ext^1(M,N).$$
\end{Prop}
This result justifies the terminology \emph{homological form} that is sometimes used for the Euler form.
We shall call a nonzero element $d=(d_1,\ldots,d_n)$ in $\mathbb{N}^n$ a \emph{dimension vector}, as it is the dimension vector of some representation in $\rep(Q)$. It is \emph{sincere} if $d_i > 0$ for $i=1,\ldots,n$; and called a \emph{root} if there exists an indecomposable representation $M$ with $d_M = d$.
The previous statement implies that $\langle d,d\rangle \le 1$;
see \cite{Kac}. If $d$ is a root with $\langle d,d \rangle=1$, then $d$ is called a \emph{real root}, and otherwise, an \emph{imaginary root}. If $d$ is a real root, then there is a unique, up to isomorphism, indecomposable representation having $d$ as a dimension vector; see \cite{Kac}. In this case, we write $M(d)$ for a fixed
indecomposable representation with this dimension vector.

A root is called a \emph{Schur root} if there exists an indecomposable representation
$M$ with $d=d_M$ and such that $M$ has a trivial endomorphism ring.
When $d$ is real, the latter condition is equivalent to  $\Ext^1(M,M)=0$.
An indecomposable representation having a trivial endomorphism ring
is called a \emph{Schur representation}.

We refer to an infinite component of the Auslander-Reiten quiver of
$\rep(Q)$ which contains projective representations
as a \emph{preprojective component} (it is unique when $Q$ is connected and non-Dynkin);
a \emph{preinjective component} is defined similarly.

A real Schur root $d$ is called \emph{preprojective} (resp.~\emph{preinjective}) if $M(d)$ lies in a preprojective (resp.~preinjective) component of the Auslander-Reiten quiver of $\rep(Q)$. Otherwise,
it is called \emph{regular}.

Suppose now that $Q$ is a Euclidean quiver. Recall that a \emph{stable tube} is a translation quiver obtained by taking the quotient of the translation quiver $\mathbb{Z}\mathbb{A}_\infty$ by some power of the translation. It is called a \emph{homogeneous tube} when that power is one. A stable tube $T$ is \emph{standard} if the full subcategory generated by the indecomposable representations in $T$ is equivalent to the path algebra $k T$ modulo the mesh relations. A representation in a stable tube which has only one arrow to and from it is called \emph{quasi-simple}. For the convenience of the reader, the following proposition collects some results concerning the structure of the Auslander-Reiten quiver of $\rep(Q)$ and the morphisms between its connected components. All the statements can be found in \cite{ASS} or \cite{AS2}.

\begin{Prop} \label{ARShape} Let $Q$ be a Euclidean quiver. The Auslander-Reiten quiver of $\rep(Q)$ contains a preprojective component $\mathcal{P}$ of shape $\mathbb{N}Q^{\rm \, op}$, a preinjective component $\mathcal{I}$ of shape $\mathbb{N}^-Q^{\rm \, op}$, a $\mathbb{P}^1(k)$-family $\{T_\lambda\}$ of standard stable tubes of which only finitely many are non-homogeneous. Moreover, for $\lambda_1, \lambda_2 \in \mathbb{P}^1(k)$, we have
\begin{enumerate}[$(1)$]
    \item $\Hom(\mathcal{I},\mathcal{P})=\Hom(\mathcal{I},T_{\lambda_1})=\Hom(T_{\lambda_1}, \mathcal{P})=0$,
\item $\Hom(T_{\lambda_1}, T_{\lambda_2})=0$ if $\lambda_1 \ne \lambda_2$.
\end{enumerate}
\end{Prop}

Recall that in the Euclidean case, the quadratic form $q(x) := \langle x, x \rangle$ is positive-semidefinite, and its radical is of rank one and is generated by a dimension vector $\delta$. In particular, an imaginary root $d$ has the property that $\langle d,d\rangle = 0$, and hence is also called \emph{isotropic}.
We shall call $\delta$ the \emph{null root} of $Q$.
Any isotropic (hence imaginary) root is a positive
multiple of $\delta$. Observe moreover that $\delta$ is a Schur root and any other positive multiple of $\delta$ is also a root, but not a Schur root.

In the Euclidean case, the \emph{regular roots} are the dimension vectors of the indecomposable representations which lie in a stable tube.
Finally, observe that if $d$ is a root, then $\langle \delta,d \rangle$ is zero (resp.~negative, positive) if and only if $d$ is regular
(resp.~preprojective, preinjective).

\section{Canonical decompositions, semi-invariants and semi-stable subcategories}

In this section, $Q$ is any acyclic quiver, unless otherwise indicated.
For a dimension vector $d=(d_1,\ldots,d_n)$, define $\rep(Q,d)$ to be the
product of vector spaces specifying matrix entries for each of the arrows of
$Q$.  An element in $\rep(Q,d)$ is canonically identified with the corresponding representation in $\rep(Q)$ of dimension vector $d$. The reductive algebraic group ${\rm GL}_d(k) := \prod_{i=1}^n{\rm GL}_{d_i}(k)$ acts on $\rep(Q,d)$ by simultaneous
change of basis and has a natural subgroup ${\rm SL}_d(k) := \prod_{i=1}^n{\rm SL}_{d_i}(k)$ which is a semi-simple subgroup. A \emph{semi-invariant} for $\rep(Q,d)$ is just a polynomial function in $k[\rep(Q,d)]$ which is invariant under the action of SL$_d(k)$. The set of semi-invariants for $\rep(Q,d)$ will be denoted ${\rm SI}(Q,d)$.

Now any linear map in $\Hom(\mathbb{Z}^n,\mathbb{Z})$ is called a \emph{weight} or a \emph{stability condition} for $Q$. For $w \in \Hom(\mathbb Z^n,\mathbb Z)$
and an element $g = (g_1,\ldots, g_n)$
in ${\rm GL}_d(k)$, we set $$w(g) = w(g_1,\ldots,g_n) = \prod_{i=1}^n{\rm det}(g_i)^{w(e_i)},$$
where for $1 \le i \le n$, the vector $e_i$ is the dimension vector of
the simple representation at vertex $i$.
In this way, when $d$ is sincere, the set of all weights for $Q$ corresponds to the set of all multiplicative characters for ${\rm GL}_d(k)$.
Following King \cite{King}, an element $f \in k[\rep(Q,d)]$ is called a \emph{semi-invariant of weight $w$} if $g(f) = w(g)f$ for all
$g = (g_1,\ldots, g_n)$
in ${\rm Gl}_d(k)$.
We write ${\rm SI}(Q,d)_{w}$ for the set of semi-invariants of weight $w$ of $\rep(Q,d)$.

When $d$ is sincere, the algebra ${\rm SI}(Q,d)$ is graded by the set of weights (which is the $\Z$-dual $K_0(Q)^*$ of $K_0(Q)$),
$${\rm SI}(Q,d) = \bigoplus_{w \; \text{weight}} {\rm SI}(Q,d)_{w}.$$
Contrary to the way it is stated in \cite{DW1}, observe that when $d$ is not sincere, the direct sum is graded instead by $K_0(Q_d)^*$, where $Q_d$ is the full subquiver of $Q$ generated by the vertices in the support of $d$. Equivalently, it is more convenient to use $K_0(Q)^*/\sim_d$ as the grading set, where for $w,w' \in K_0(Q)^*$, we have $w \sim_d w'$ if $w,w'$ agree on the elements supported over $Q_d$. In other words, if $f$ is a semi-invariant of weight $w$ in ${\rm SI}(Q,d)$ and $w \sim_d w'$, then $f$ is also a semi-invariant of weight $w'$, even if $w'\ne w$. In the sequel, for a representation $M$, we sometimes write $\sim_M$ for $\sim_{d_M}$. We refer the reader to \cite{DW1, Sch2} for fundamental results
on semi-invariants of quivers. Let us just recall the results we need. Given two representations $M,N$, we have an exact sequence
\begin{equation} \label{eqn1}
0 \to \Hom(M,N) \to \Hom(P_0,N) \stackrel{f^M_N}{\longrightarrow} \Hom(P_1,N) \to \Ext^1(M,N) \to 0
\end{equation}
where $0 \to P_1 \to P_0 \to M \to 0$ is the canonical projective resolution of $M$.
We see that the
$k$-linear map $f^M_N$ corresponds to a square matrix if and only if $\langle d_M,d_N \rangle = 0$. In this case, we define $C^M(N)$ to be the determinant
of $f^M_N$. Then $C^M(-)$, which is well defined up to a nonzero scalar, is a semi-invariant of weight $\langle d_M, - \rangle$ for $\rep(Q,d_N)$. We see that $C^M(N)\ne 0$ if and only if
$\Hom(M,N)=0=\Ext^1(M,N)$.
The following result was proven in \cite{DW3} for the general case and in \cite{SVDB} for the characteristic zero case.

\begin{Prop} \label{fund}
Let $w$ be a weight and $d$ a dimension vector. Then the vector space ${\rm SI}(Q,d)_{w}$ is generated over $k$ by the semi-invariants $C^M(-)$ where $\langle d_M,d\rangle=0$
and $w \sim_d \langle d_M, - \rangle$.
\end{Prop}

Now, let us introduce the main object of study of this paper. Let $w$ be a weight. A representation $M$ in $\rep(Q)$ is said to be
\emph{semi-stable with respect to $w$} or \emph{$w$-semi-stable} if there exist $m>0$ and $f \in {\rm SI}(Q,d_M)_{mw}$ such that $f(M) \ne 0$.
Given an element $d \in \Z^n$, write
$\rep(Q)_{d}$
for the full subcategory of $\rep(Q)$ of the semi-stable representations with respect to the weight $\langle d, - \rangle$. We warn the reader here not to confuse $\rep(Q)_d$ with $\rep(Q,d)$, when $d$ is a dimension vector. These subcategories will be referred to as the \emph{semi-stable subcategories} of $\rep(Q)$.  If $Q'$ is a full subquiver of $Q$ and $w$ is a weight for $Q$, then we denote by $w|_{Q'}: K_0(Q') \to \Z$ the weight of $Q'$ which is the restriction of $w$. Given $M \in \rep(Q)$, denote by ${\rm Su}(M)$ its support and by $Q_{d_M}$ the full subquiver of $Q$ generated by Su$(M)$. The following observation will
be handy in the sequel.
\begin{Cor} \label{cor1} Let $M \in \rep(Q)$ and $d\in \Z^n$. Then $M \in \rep(Q)_d$ if and only if there exist a positive integer $m$ and a representation
$V$ supported over $Q_{d_M}$ with $\langle d_V, - \rangle  \sim_{M} \langle md, - \rangle$ and with $C^{V}(M) \ne 0$.
\end{Cor}
\begin{proof} For the sufficiency, assume there exist a positive integer $m$ and a representation
$V$ supported over $Q_{d_M}$ with $\langle d_V, - \rangle \sim_{M} \langle md, - \rangle$ and with $C^{V}(M) \ne 0$. Since $ \langle d_V, - \rangle \sim_{M} \langle md, - \rangle$, we see that $\langle md, - \rangle$ is a weight for $C^{V}(-)$ in SI$(Q,d_M)$. Since $C^{V}(M)\ne 0$, we see that $M$ is $\langle d, - \rangle$-semi-stable. Assume now that $M \in \rep(Q)_d$.  Then $M$ is identified with a representation $M'$ of $Q_{d_M}$ and clearly, $M'$ is $d'$-semi-stable in $\rep(Q_{d_M})$, where $d'= \langle d, - \rangle \rest_{Q_{d_M}}$. Then, there exist $m > 0$ and $f \in {\rm SI}(Q_{d_M},d_M)_{md'}$
such that $f(M') \ne 0$. By Proposition \ref{fund}, there exists a representation $V' \in \rep(Q_{d_M})$ with $C^{V'}(M') \ne 0$ and $\langle d_{V'}, - \rangle \rest_{Q_{d_M}} = md'$. Now, $V'$ is identified with a representation $V$ of $\rep(Q)$ and we see that $C^V(M) \ne 0$.
\end{proof}

Now, consider $K_0(\rep(Q))\otimes_\Z \mathbb{Q}$, which will be identified with $\mathbb{Q}^n$. Given $d \in \mathbb{Q}^n$, let $[d]$ be the ray containing $d$ and emerging from the origin.
We will denote by $(\mathbb{Q}^n)_+$ the
the positive orthant of $\mathbb{Q}^n$.
For $d \in \Z^n$, we define $H_d$ to be the set of all elements $f$ in $\mathbb{Q}^n$ with $\langle f , d \rangle = 0$ (observe that $H_d$ is a union of rays since $\langle -, d \rangle$ is homogeneous of degree one).
We denote by $H^{+}_d$ the intersection of $H_d$ with $(\mathbb{Q}^n)_+$.
When $Q$ is a Euclidean quiver, one particularly important such set is $H_{\delta}$ defined by the
equation $\langle -,\delta \rangle = 0$ (or equivalently, of equation $\langle \delta, - \rangle = 0$). As observed above, the roots in
$H^{+}_\delta$ are precisely the regular roots.

Now, let us recall the concept of canonical decomposition
of a dimension vector defined by Kac \cite{Kac0};
see also \cite{DW1, Kac}.
Let $d$ be a dimension vector.
Suppose that there exists a non-empty open set $\mathcal{U}$ in $\rep(Q,d)$ and dimension vectors $d(1),\ldots,d(r)$ satisfying the following property:
for each $M \in \mathcal{U}$, there exists a decomposition
$$M \cong M_1 \oplus \cdots \oplus M_r$$
where each $M_i$ is indecomposable of dimension vector $d(i)$.  (Note that
the isomorphism class of $M_i$ may vary depending on the choice of $M \in
\mathcal U$ --- all that is fixed is the dimension vector.)
In this case, we write
$$d = d(1) \oplus \cdots \oplus d(r)$$
and call this expression the \emph{canonical decomposition} of $d$.
The canonical decomposition always exists and is unique; see \cite{Kac0}. For more details concerning the canonical decomposition, we refer the reader to \cite{DW1}. Let us recall two fundamental results due to Kac and Schofield. Given two dimension vectors $d_1,d_2$, we set
$${\rm ext}(d_1,d_2) = \min(\{{\rm dim}_k\Ext^1(M_1,M_2) \mmid d_{M_1} = d_1, d_{M_2}=d_2\}).$$ One also defines ${\rm hom}(d_1,d_2)$ in an analogous way. The following result is due to Kac \cite{Kac}.
\begin{Prop}[Kac]\label{Kac1}
A decomposition $$d = d(1) + \cdots + d(m),$$
where the $d(i)$ are dimension vectors corresponds to the canonical decomposition for $d$ if and only if the $d(i)$ are all Schur roots and ${\rm ext}(d(i),d(j))=0$ for $i \ne j$.
\end{Prop}
Actually, using the previous notations, a stronger statement holds; see \cite{Kac}.  There exists a non-empty open set $\mathcal{U}'$ in $\rep(Q,d)$ such that for any $M \in \mathcal{U}'$, $$M \cong M_1 \oplus M_2 \oplus \cdots \oplus M_r$$ where each $M_i$ is a Schur representation of dimension vector $d(i)$ and we have $\Ext^1(M_i,M_j)=0$ if $i\ne j$. A representation $M$ which decomposes as a direct sum of pairwise Ext-orthogonal Schur representations will be called a \emph{general representation of dimension vector $d$}.  Hence, given a general representation $M$ of dimension vector $d$, one can recover the canonical decomposition of $d$ by looking at the dimension vectors of the indecomposable summands of $M$.

\medskip

A dimension vector $d$ whose canonical decomposition only involves real Schur roots is called \emph{prehomogeneous}.  In such a case, there is a unique, up to isomorphism, general representation of dimension vector $d$. A fixed such representation will be denoted by $M(d)$ in the sequel.  Note that $\Ext^1(M(d),M(d))=0$. When $d$ is a real Schur root, this notation agrees with our previous notation.

\medskip

The following result, due to Schofield \cite{Sch}, shows how the canonical decomposition behaves when we multiply a dimension vector by a positive integer.  For a Schur root $d$ and a positive integer $r$, let us define the following
notation:
$$d^r := \left\{ \begin{array}{ll} d\oplus \cdots \oplus d\, (r\, \text{copies}), & \text{if $d$ is real or isotropic}\\
rd \, , & \text{otherwise}.   \end{array}\right.$$
\begin{Prop}[Schofield] \label{Sch}
Let $d$ be a dimension vector with canonical decomposition $d = d(1) \oplus \cdots \oplus d(r)$. For $m$ a positive integer, $md$ has canonical decomposition $md = d(1)^m \oplus \cdots \oplus d(r)^m.$
\end{Prop}

Observe that when $Q$ is of wild type and $s$ is a Schur imaginary non-isotropic root, then all the positive multiples of $s$ are also Schur imaginary roots. When $Q$ is a Dynkin or a Euclidean quiver, two distinct Schur roots lie on different rays in $\mathbb{Q}^n$.

Given a ray $r \in \mathbb{Q}^n$ with a dimension vector $d$ in it, one can consider the distinct rays $r_1, \ldots, r_s$ corresponding to the indecomposable summands of the canonical decomposition of $d$, and this is well defined by the above result of Schofield. In fact, we have the following.

\begin{Lemma} Let $d_1, d_2$ be two Schur roots and suppose that $d_1' \in [d_1], d_2' \in [d_2]$ are Schur roots.  Then ${\rm ext}(d_1, d_2)=0$ if and only if ${\rm ext}(d_1', d_2')=0$.
\end{Lemma}

\begin{proof} We need only to prove the necessity and further, we need only to consider the case where at least one of the $d_i$ is imaginary and non-isotropic, say $d_1$. So assume that ${\rm ext}(d_1, d_2)=0$. Consider $d = d_1 + d_2 = d_1 \oplus d_2$. We have to show that
$d_1' \oplus d_2'$
is the canonical decomposition of $f := d_1' + d_2'$. Let the canonical decomposition of $f$ be
$$f = f(1) \oplus \cdots \oplus f(t).$$
There are positive integers $p_1, p_2, r_1, r_2$ such that $p_1d_1=r_1d_1'$ and $p_2d_2 = r_2d_2'$.  Let $\ell$ be the least common multiple of the $p_i$ and the $r_i$. From Proposition \ref{Sch}, we know that the canonical decomposition of $\ell f$ is
$$\ell f = f(1)^\ell \oplus \cdots \oplus f(t)^\ell.$$
Observe now that \begin{eqnarray*} \ell f & = & \ell d_1' + \ell d_2'\\
& = & t_1 d_1 +  t_2 d_2
\end{eqnarray*}
for some positive integers $t_1,t_2$. Using the canonical decomposition of $d$, we see that there are indecomposable Schur representations $M_1,M_2$, that are pairwise Ext-orthogonal such that $d_{M_j} = d_j$. Consider first the case where $d_2 = d_2'$ is either real or isotropic.
Using these Schur representations, we see that the representations
$M_1^{t_1}, M_2$
are pairwise Ext-orthogonal. Hence, using the definition of ext, we see that the Schur roots $$\{t_1 d_1 = \ell d_1', d_2\}$$ are ext-orthogonal.
Hence, from Proposition \ref{Kac1}, the canonical decomposition of $\ell f$ is also given by
$$(\ell d_1') \oplus d_2^{t_2}.$$
By uniqueness of the canonical decomposition, we have
$$f(1)^\ell \oplus \cdots \oplus f(t)^\ell = (\ell d_1') \oplus d_2^{t_2},$$
from which it follows that $(\ell d_1') = (\ell f(j))$ for a unique $j$ and $d_2 = f(i)$ for some $i$. This implies that $d_1' = f(j)$ and $d_2' = f(i)$. In particular, we get ${\rm ext}(d_1', d_2')=0$. Assume now that $d_2$ is imaginary and non-isotropic. The representations
$M_1^{t_1}, M_2^{t_2}$
are pairwise Ext-orthogonal. Hence, we see that the Schur roots $$\{t_1 d_1 = \ell d_1', t_2d_2 = \ell d_2'\}$$ are ext-orthogonal.
Using the same argument as above, we see that $(\ell d_1') = (\ell f(j))$ for a unique $j$ and $(\ell d_2') = (\ell f(i))$ for a unique $i$ . This implies that $d_1' = f(j)$ and $d_2' = f(i)$. In particular, we also get ${\rm ext}(d_1', d_2')=0$.
\end{proof}

Now, let $r_1, r_2$ be two rays in $(\mathbb{Q}^n)_+$, each containing a Schur root. Then the condition ${\rm ext}(r_1,r_2)=0$ does not depend on the particular choice of a root in each ray, and hence is well defined. Thus, we can make $(\mathbb{Q}^n)_+$ into a simplicial fan whose rays are the rays associated to the Schur roots. Two rays $r_1, r_2$ associated to Schur roots belong to the same simplicial cone if and only if ${\rm ext}(r_1, r_2)= {\rm ext}(r_2, r_1) = 0$. It is then clear that if $d$ is a dimension vector belonging to the relative interior of a simplicial cone generated by the rays $r_1, \ldots, r_s$, then the rays associated to the Schur roots in the canonical decomposition of $d$ are precisely $r_1, \ldots, r_s$. We call this simplicial fan on $(\mathbb{Q}^n)_+$ the \emph{cd-fan}.

We would like to extend the structure of simplicial fan on $(\mathbb{Q}^n)_+$ to a structure of simplicial fan on the whole $\mathbb{Q}^n$. A similar construction is done in \cite{IOTW}, but with less generality since not all imaginary Schur roots are considered. Let us give the construction. We denote by $p_i$ the dimension vector of the indecomposable projective representation at vertex $i$, $i=1,2, \ldots, n$. We add the rays $[-p_i]$ for $1 \le i \le n$ to the cd-fan. Let us define a compatibility relation on the set $\{[d] \mid d \; \text{is a Schur root}\,\}\cup\{[-p_i] \mid i=1,2,\ldots,n\}$ of rays. If $d_1, d_2$ are Schur roots, then $[d_1], [d_2]$ are compatible if ${\rm ext}([d_1],[d_2])=0$. If $d_1$ is a Schur root, then $[d_1],[-p_i]$ are compatible if $d_1$ is not supported at $i$. Finally, $[-p_i],[-p_j]$ for $i \ne j$ are always compatible.
In this way, we can make $\mathbb{Q}^n$ into a simplicial fan such that if $d \in \Z^n$, then $d$ lies in the relative interior of a simplicial cone generated by some rays $r_1, \ldots, r_s$. The following proposition is essential for the definition of canonical presentation.

\begin{Prop}[\cite{IOTW}] \label{PropCanPres}
Let $d \in \Z^n$.
There exists a decomposition
$$d = t_1d_1 + t_2d_2 + \cdots + t_md_m + t_{m+1}d_{m+1} + \cdots + t_sd_s$$
where the $t_i$ are positive integers, $d_1, \ldots, d_m$ are Schur roots, and $d_{m+1}, \ldots, d_s$ are negative of dimension vectors of indecomposable projective representations. Moreover, the $d_i$, $1 \le i \le s$, are compatible in the sense defined above and $t_i=1$ whenever $d_i$ is an imaginary non-isotropic Schur root.
\end{Prop}

The decomposition given in Proposition \ref{PropCanPres} is called the \emph{canonical presentation} of $d$. The positive part $t_1d_1 + t_2d_2 + \cdots + t_md_m$ is denoted $d_+$ while the negative part $ t_{m+1}d_{m+1} + \cdots + t_sd_s$ is denoted $d_-$.
Observe that $d_+$ is a dimension vector whose canonical decomposition is $d_+ = t_1d_1 + \cdots + t_md_m$ (or $d_+ = d_1^{t_1} \oplus \cdots \oplus d_m^{t_m}$ using Schofield's notation) and $d_-$ is equal to $-d_P$, where $P$ is a projective representation whose top has support disjoint to $d_+$.

\begin{Exam}
Let $Q$ be the quiver $2 \rightarrow 1$. The Schur roots are the dimension vectors $(0,1), (1,0)$ and $(1,1)$, while the negative of the dimension vectors of the indecomposable projective representations are $(-1,0), (-1,-1)$. The figure below illustrates the fan for $Q$ that is defined above. The simplicial cones generated by two rays are denoted $C_1, C_2, C_3, C_4$ and $C_5$. The vector $d=(0,-1)$, shown in the diagram, lies in the relative interior of $C_5$. Therefore, we see that the canonical presentation of $d$ is $d = (1,0)+(-1,-1)$, where both coefficients are understood to be ones.
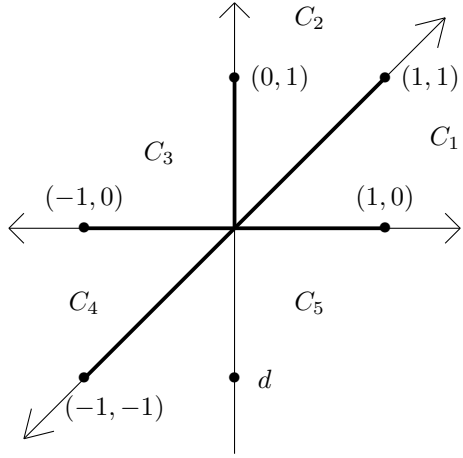
\begin{figure}[h]
  \centering
  \begin{tikzpicture}[xscale=2,yscale=2]

    \draw (-1.5,0) -- (1.5,0);
\draw (1.4,0.1) -- (1.5,0);
\draw (-1.4,-0.1) -- (-1.5,0);
\draw (-1.4,0.1) -- (-1.5,0);
\draw (1.4,-0.1) -- (1.5,0);
    \draw (0,-1.5) -- (0,1.5);
\draw (-0.1,1.4) -- (0,1.5);
\draw (0.1,1.4) -- (0,1.5);
\draw[line width=0.5mm] (-1,-1) -- (1,1);
\draw (-1.4,-1.4) -- (1.4,1.4);
\draw (-1.35,-1.2) -- (-1.4, -1.4);
\draw (-1.2,-1.35) -- (-1.4, -1.4);
\draw (1.35,1.2) -- (1.4, 1.4);
\draw (1.2,1.35) -- (1.4, 1.4);
\draw[line width=0.5mm] (-1,0) -- (1,0);
\draw[line width=0.5mm] (0,0) -- (0,1);
\node at (0.2,-1) {$d$};

   \node at (1, 0.2) {$(1,0)$};
   \node at (-1, 0.2) {$(-1,0)$};
    \node at (0.3, 1) {$(0,1)$};
\node at (1.3,1.0) {$(1,1)$};
\node at (-0.8,-1.2) {$(-1,-1)$};

\node at (1, 0) {$\bullet$};
   \node at (-1, 0) {$\bullet$};
    \node at (0, 1) {$\bullet$};
\node at (1,1) {$\bullet$};
\node at (-1,-1) {$\bullet$};
\node at (0,-1) {$\bullet$};

\node at (1.4, 0.6) {$C_1$};
   \node at (0.5,1.4) {$C_2$};
    \node at (-0.5,0.5) {$C_3$};
\node at (-1,-0.5) {$C_4$};
\node at (0.5, -0.5) {$C_5$};

  \end{tikzpicture}
  \caption{The above-defined fan for the quiver of type $\mathbb{A}_2$}
\label{fig:A2}
\end{figure}

\end{Exam}

The following lemma summarizes what we just discussed. It follows from our discussion and what appears in \cite{IOTW}.

\begin{Lemma}\label{IOTW} For an acyclic quiver $Q$, the above-defined structure makes $\mathbb{Q}^n$ into a simplicial fan of dimension $n-1$. If $d \in \Z^n$ lies in the relative-interior of a simplicial cone generated by the rays $r_1, \ldots, r_s$, then let $d_i \in r_i$ be the smallest dimension vector.
\begin{enumerate}[$(1)$]
\item The vectors $d_1, \ldots, d_s$ are linearly independent.
\item The canonical presentation of $d$ is $d = \sum_{i=1}^st_id_i$, where the $t_i$ are positive integers.
\item The sum splits into two parts, where the summands in the first part add up to $d_+$, and the summands in the second part add up to $d_-$.
\end{enumerate}
\end{Lemma}

Note that in the above notation, if $d_i$ is imaginary and non-isotropic, then $t_id_i$ means the Schur root $(t_id_i)$ with a coefficient $1$ in front. We will refer to this simplical fan on $\mathbb{Q}^n$ as the \emph{cp-fan}. Observe that if we restrict this simplicial fan to the positive orthant, then we get the cd-fan.

\section{Exceptional objects and exceptional sequences}

A representation $M$ in $\rep(Q)$ with $\Ext^1(M,M)=0$ is said to be \emph{rigid}. If it is also indecomposable,
it is called \emph{exceptional}. It is well known (see for example \cite{DW1}) that a representation $M$ is rigid if and only if
its orbit in $\rep(Q,d_M)$ is open.
A sequence of exceptional representations $(E_1,\dots,E_r)$ is called an
\emph{exceptional sequence} if for all $i<j$, we have
$\Hom(E_i,E_j)=0=\Ext^1(E_i,E_j)$. Note that some authors define an exceptional sequence using the ``upper triangular'' convention rather than the ``lower triangular'' convention that is used here.

The maximum length of an exceptional sequence in $\rep(Q)$ is the number
of vertices in $Q$.  Any exceptional sequence can be extended to one of maximal length; see \cite{CB}.

Given any full subcategory $\mathcal{V}$ of $\rep(Q)$, which may consist of a single object, we define $\mathcal{V}^\perp$ to be the full subcategory of $\rep(Q)$
generated by the representations $M$ for which $\Hom(V,M) = \Ext^1(V,M) = 0$ for every $V \in \mathcal{V}$. We call $\mathcal{V}^\perp$ the \emph{right orthogonal category}
to $\mathcal{V}$.  One also has the dual notion of \emph{left orthogonal category} to $\mathcal{V}$, denoted $^\perp\mathcal{V}$. We denote by $\C(\mathcal{V})$ the smallest abelian, extension-closed subcategory of $\rep(Q)$ generated by the objects in $\mathcal{V}$.
If $E$ is an exceptional sequence, and $\add E$ its additive hull,
we will abbreviate $\C(\add E)$ by
$\C(E)$.
We need the following well known fact. It was first proven by Geigle and Lenzing in \cite{GL}, and later by Schofield in \cite{Sch2}.

\begin{Prop}\label{GL}
Let $X$ be an exceptional representation in $\rep(Q)$, where $Q$ is an acyclic quiver. Then the category $X^\perp$ $($or $\p X)$ is equivalent to $\rep(Q')$ where $Q'$ is an acyclic quiver
having $|Q_0|-1$ vertices.
\end{Prop}

As an easy consequence, we have the following.

\begin{Prop} \label{GL2} Let $E=(X_1,\ldots,X_r)$ be an exceptional sequence in $\rep(Q)$ where $Q$ is an acyclic quiver.  Then $\C(E)^\perp$ $($or $\p \C(E))$ is equivalent to $\rep(Q')$ where $Q'$ is an acyclic quiver
having $|Q_0|-r$ vertices.
\end{Prop}

The following lemma is probably well known.  We include a proof for completeness.

\begin{Lemma} \label{Lemma1}
Let $Q$ be a Euclidean quiver and $V,X$ be indecomposable representations with $X$ non-exceptional.
\begin{enumerate}[$(1)$]
    \item If $V$ is preprojective, then $\Hom(V,X) \ne 0$.
    \item If $V$ is preinjective, then $\Ext^1(V,X) \ne 0$.
\end{enumerate}
\end{Lemma}

\begin{proof}
We only prove the second part. Suppose that $V$ is preinjective. There exists a non-negative integer $r$ with $\tau^{-r}V$ injective
indecomposable.
Hence, by the Auslander-Reiten formula,
$$\Ext^1(V,X) \cong D\Hom(X,\tau V) \cong D\Hom(\tau^{-r-1}X,\tau^{-r} V).$$
Being non-exceptional, $X$ lies in a stable tube of the Auslander-Reiten quiver of $\rep(Q)$, and on or above the level in the tube consisting of
representations whose dimension vector is the null root.  It follows
that the entire $\tau$-orbit of $X$ is sincere.
Hence, $\tau^{-r-1}X$ is sincere. Since $\tau^{-r}V$ is injective indecomposable,
we clearly have $\Hom(\tau^{-r-1}X,\tau^{-r} V) \ne 0$.
\end{proof}

\begin{Prop} \label{equiv_thick}
Let $\C$ be an abelian, extension-closed subcategory of $\rep(Q)$.  The following are equivalent.
\begin{enumerate}[$(a)$]
    \item $\C$ is generated by an exceptional sequence,
\item $\C$ is equivalent to $\rep(Q')$ for an acyclic quiver $Q'$,
\item $\C = V^\perp$ for some rigid representation $V$ of $\rep(Q)$.
\end{enumerate}
\end{Prop}

\begin{proof}
Suppose that $\C$ is generated by an exceptional sequence $E= (X_{r+1},\ldots,X_n)$.  Then, as was already remarked, $E$ can be completed to a full exceptional sequence $$E' = (X_1,\ldots,X_r,X_{r+1},\ldots,X_n).$$   Then $\C = \C(E'')^\perp$ where $E'' = (X_1,\ldots,X_r)$ is an exceptional sequence.  By Proposition \ref{GL2}, we get that $\cap_{i=1}^rX_i^\perp = \C(E'')^\perp$ is equivalent to $\rep(Q')$ for an acyclic quiver $Q'$. This proves that $(a)$ implies $(b)$.  The fact that $(b)$ implies $(a)$ follows from the fact that the module category of any triangular algebra is generated by an exceptional sequence, and an exceptional sequence for $\rep(Q')$ is also an exceptional sequence for $\rep(Q)$, since $\rep(Q')$ is an exact extension-closed subcategory of $\rep(Q)$.

Suppose now that $\C$ is generated by an exceptional sequence $E= (X_{r+1},\ldots,X_n)$.  As done in the first part of the proof, $E$ can be completed to a full exceptional sequence $E' = (X_1,\ldots,X_r,X_{r+1},\ldots,X_n)$. Then $E'' = (X_1,\ldots,X_r)$ is an exceptional sequence and $\C(E'')$ is equivalent to $\rep(Q')$ for an acyclic quiver $Q'$.  Take $V$ a minimal projective generator of $\C(E'')$.  Then $V$ is rigid and $\C(V) = \C(E'')$ with $V^\perp = \C(E'')^\perp = \C$. This proves that $(a)$ implies $(c)$. Finally, assume that $\C = V^\perp$ for a rigid representation $V$. Then $V$ is a partial tilting module and by Bongartz's lemma, it can be completed to a tilting module $V \oplus V'$.  Then ${\rm End}(V \oplus V')$ is a tilted algebra and hence has no oriented cycles in its quiver. In particular, the non-isomorphic indecomposable summands of $V$ form an exceptional sequence $E''=(Y_1,\ldots,Y_s)$ that can be completed to a full exceptional sequence $E' = (Y_1,\ldots,Y_s,Y_{s+1},\ldots,Y_n)$.  Then $V^\perp = \C(E)$, where $E = (Y_{s+1},\ldots,Y_n)$. This proves that $(c)$ implies $(a)$.
\end{proof}

We can refine the previous proposition in the case that $Q$ is
Euclidean.  Recall our definition that a \emph{possibly disconnected Euclidean}
quiver is a quiver which has one Euclidean component and a finite number of
Dynkin components. Similarly a \emph{possibly disconnected Dynkin} quiver is a quiver whose connected components are all of Dynkin type.

\begin{Lemma}\label{kindofperp} Let $Q$ be a Euclidean quiver and $V$ be a rigid
representation.  Then $V^\perp$ is equivalent to the representations of
a (possibly disconnected) Euclidean quiver  if $V$ is regular; otherwise, $V^\perp$ is equivalent to
the representations of a (possibly disconnected) Dynkin quiver. \end{Lemma}

\begin{proof}  If $V$ is regular, then $V^\perp$ will include all the
objects from the homogeneous tubes, and in particular, it will include
representations whose dimension vector is the unique imaginary Schur root.
Since we know $V^\perp$ is equivalent to
the representations of some quiver, it must be equivalent to the representations
of a (possibly disconnected) Euclidean quiver.  (It clearly cannot be equivalent to the
representations of a wild quiver, since it is contained in
$\rep(Q)$.)

If $V$ is not regular, then let $Y$ be a non-regular indecomposable
summand of $V$.  Applying Lemma \ref{Lemma1}, we see that $Y^\perp$
contains no non-exceptional indecomposables.  Thus $V^\perp \subseteq
Y^\perp$ must be of finite representation type.
\end{proof}

\section{Description of semi-stable subcategories ---
proof of Theorem \ref{one}}

In this section, we prove Theorem \ref{one}. We begin with $Q$ any
acyclic quiver, and then specialize to the Euclidean case.
Recall the following from \cite{King}.

\begin{Prop}[King] \label{King}
Let $d \in \Z^n$.  Then $M \in \rep(Q)$ is $\langle d, - \rangle$-semi-stable if $\langle d, d_M \rangle = 0$ and whenever $M'$ is a subrepresentation of $M$, then $\langle d, d_{M'} \rangle \le 0$.
\end{Prop}

The following proposition is well known, but we include a proof
of our own.

\begin{Prop} \label{dnonreg}
Suppose that $d$ is a prehomogeneous dimension vector
and set $V = M(d)$.
Then
$\rep(Q)_{d} = \{M \in \rep(Q) \mmid C^V(M) \ne 0\} = V^\perp.$
\end{Prop}

\begin{proof}
First, the fact that  $\{M \in \rep(Q) \mmid C^V(M) \ne 0\}= V^\perp$ follows from the exact sequence \eqref{eqn1} and the definition of $C^V(M)$. Let $M \in \rep(Q)_d$. By Proposition \ref{King}, $\langle d, d_M \rangle =0$. In particular dim$_k\Hom(V,M) = {\rm dim}_k\Ext^1(V,M)$. If both spaces are non-zero, then there exists a subrepresentation $M'$ of $M$ with an epimorphism $V \to M'$. This yields an epimorphism $\Ext^1(V,V) \to \Ext^1(V,M')$. Since $V$ is rigid, this gives $\Ext^1(V,M')=0$. But then, $\langle d, d_{M'} \rangle = {\rm dim}_k\Hom(V,M') > 0$, in contradiction to Proposition \ref{King}. Hence, $\Hom(V,M) = 0 = \Ext^1(V,M)$ which means that $M \in V^\perp.$ Conversely, it is clear that if $M \in \rep(Q)$ with $C^V(M) \ne 0$, then
$M \in \rep(Q)_d$.
\end{proof}

Let us denote by $\mathcal{R}{\rm eg}$ the full additive subcategory of $\rep(Q)$ generated by the indecomposable regular representations.
When $Q$ is a Euclidean quiver, the canonical decomposition of a dimension vector only involves real Schur roots and possibly the null root $\delta$.
We first need the following simple lemma about canonical decompositions.

\begin{Lemma}\label{onlyreg} Let $Q$ be a Euclidean quiver. Let $d$ be a dimension vector which has the null root in
its canonical decomposition. Then all the other summands of its
canonical decomposition are regular.  \end{Lemma}

\begin{proof} Suppose that some real Schur root $\alpha$ appears in
the canonical decomposition of $d$.
Seeking a contradiction,
suppose that $M(\alpha)$ is preinjective.
We know by Proposition \ref{Kac1}
that there exists an indecomposable with dimension vector $\delta$,
say $V$, such
that $\Ext^1(V,M(\alpha))=0=\Ext^1(M(\alpha),V)$.
On the other hand, Lemma \ref{Lemma1} applied to $M(\alpha)$ and $V$
implies that $\Ext^1(M(\alpha),V)\ne 0$. A similar contradiction
follows dually if $M(\alpha)$ is preprojective.  Therefore $M(\alpha)$ is
regular.  \end{proof}

In the sequel, for a dimension vector $d$, we will denote by $\hat d$ the sum of the real Schur roots appearing in its canonical decomposition. The following
result complements Proposition \ref{dnonreg} in the Euclidean case.

\begin{Prop}\label{dreg}
Let $Q$ be a Euclidean quiver. Suppose that $d$ is a dimension vector with canonical decomposition
$d = \hat d \oplus \delta^s$
with $s \ne 0$. Set $V = M(\hat d)$.
Then
$$\rep(Q)_{d} = \{M \in \rep(Q) \mmid C^V(M) \ne 0\} \cap \mathcal{R}{\rm eg} = V^\perp \cap \mathcal{R}{\rm eg}.$$
\end{Prop}

\begin{proof}
By Proposition \ref{dnonreg}, we only need to prove the first equality. Let $M(\delta)$ be a general representation of dimension vector $\delta$. Then $V \oplus M(\delta)^s$ is a general representation of dimension vector $d$. By Lemma \ref{onlyreg}, we see that $V$ is regular.
Let $M \in \rep(Q)_{d}$ be indecomposable.  Then $\langle d, d_M \rangle = 0$. If $M$ is preinjective, then $\Ext^1(M(\delta) \oplus V, M) = 0$, and hence $0 = \langle d, d_M \rangle = {\rm dim}_k \Hom(V \oplus M(\delta)^s, M)$. But $\Hom(M(\delta),M) \ne 0$ by Lemma \ref{Lemma1}, so we get a contradiction. Similarly, $M$ cannot be preprojective. Therefore, $M$ is regular. We may assume that $\Hom(M(\delta), M) = 0 = \Ext^1(M(\delta),M)$. This gives $\langle d, d_M \rangle = {\rm dim}_k \Hom(V,M) - {\rm dim}_k\Ext^1(V,M)=0$.  Assume $\Hom(V,M) \ne 0$. There exists an epimorphism $V \to L$ for some subrepresentation $L$ of $M$. Since we have an epimorphism $\Ext^1(V, V) \to \Ext^1(V, L)$ and $V$ is rigid, we see that $\Ext^1(V,L)=0$. Since $\langle d, d_{L} \rangle \le 0$, this gives $\Hom(V,L)=0$, a contradiction. Therefore, $\Hom(V,M)=0$ and hence, $M \in V^\perp.$ This shows $\rep(Q)_{d} \subseteq V^\perp \cap \mathcal{R}{\rm eg}$.
Conversely, assume that $M \in V^\perp \cap \mathcal{R}{\rm eg}$.
First, $M \in V^\perp$ is clearly equivalent to
$C^V(M) \ne 0$. Since $M$ is regular, $M$ lies in a stable tube.
Hence, there exists a quasi-simple representation $W$ lying
in a homogeneous tube (and hence having dimension vector $\delta$) with $\Hom(W,M) = \Ext^1(W,M) = 0$.  We therefore have $C^{V
\oplus W^{s}}(M) \ne 0$, which shows that $M \in \rep(Q)_d$.
\end{proof}

We can now prove our first theorem from the introduction. Recall that for a possibly disconnected Euclidean quiver $Q$, the indecomposable \emph{regular representations} are defined to be the regular representations of the Euclidean component together with the indecomposable representations of the Dynkin components.

\begin{Theo}\label{one} For $Q$ a Euclidean quiver, an abelian and extension-closed
subcategory $\mathcal B$ of $\rep (Q)$ is the subcategory of
$\theta$-semi-stable representations for some $\theta$ if and only if either:
\begin{enumerate}[$(i)$]
\item $\mathcal B$ is finitely generated, or
\item  there exists some abelian, extension-closed,
finitely generated subcategory $\mathcal A$ of
$\rep(Q)$, equivalent to the representations of a Euclidean quiver (possibly
disconnected),
and $\mathcal B$ consists of all the regular objects of $\mathcal A$.
\end{enumerate}
\end{Theo}

\begin{proof}
First, write $\theta = \langle d, - \rangle$. If $d$ is not a dimension vector, then it follows from Corollary \ref{cor1} that $\rep(Q)_d$ is supported over a proper subquiver of $Q$, and hence it is an abelian extension-closed subcategory of a category of representations of a (possibly disconnected) Dynkin quiver and hence it satisfies (i). So assume that $d$ is a dimension vector.
Suppose the canonical decomposition
of $d$ does not include the null root.  Then by Proposition \ref{dnonreg},
we know that $\rep(Q)_d$ is of the form $V^\perp$ for some rigid
representation $V$.  By Proposition \ref{equiv_thick}, this implies
that it is finitely generated.

Suppose the canonical decomposition of $d$ does include the null root.
By Proposition \ref{dreg}, we know that
the
semi-stable category corresponding to $d$ is of the form
$V^\perp \cap \reg$.  By Lemma \ref{onlyreg}, we can take $V$ to be
regular and rigid.   By Lemma \ref{kindofperp}, $V^\perp$ is equivalent to
the representations of a (possibly disconnected) Euclidean quiver.
The subcategory is therefore of type (ii).

Conversely, suppose that we have a finitely generated abelian and extension-closed subcategory
$\mathcal B$ of $\rep(Q)$.  It is generated by an exceptional sequence, say
$(X_{r+1},\dots,X_n)$ which can be extended to a full exceptional
sequence $(X_1,\dots,X_n)$.  Let $\mathcal F$ be the extension-closed abelian subcategory
generated by $(X_1,\dots,X_r)$, and let $P$ be a projective generator
of $\mathcal F$.
Then $\mathcal B = P^\perp$.  Since $P$ is partial tilting, it is the
general representation of dimension $d_P$.  Therefore the semi-stable subcategory
associated to $d_P$ is $\mathcal B$ by Proposition \ref{dnonreg}.

Finally, suppose that we have an extension-closed abelian subcategory $\mathcal B$ as in
(ii), which consists of the regular objects in
some finitely generated abelian and extension-closed
subcategory $\mathcal A$, with $\mathcal A$ equivalent to the representations
of some (possibly disconnected) Euclidean quiver.  As in the previous case, we know that
$\mathcal A$ can be written as $P^\perp$.   Since $\mathcal A$ is of
(possibly disconnected) Euclidean type, $P$ must be regular by Lemma \ref{kindofperp}.  Hence, an indecomposable summand of $P$ is either in a stable tube or is supported by a Dynkin component of the quiver. By
Proposition \ref{Kac1}, we know that a general representation of dimension
vector $\delta+d_P$ will be isomorphic to a direct sum of $P$ and
an indecomposable representation of dimension vector $\delta$.
Therefore the semi-stable subcategory associated to $\delta+d_P$ will
be $\mathcal B$ by Proposition \ref{dreg}.
\end{proof}

Note that from the previous theorem, not all abelian and extension-closed subcategories of $\rep(Q)$ are semi-stable, and that a semi-stable subcategory may be disconnected. In \cite{Dichev}, there is a description of all the \emph{connected} abelian and extension-closed subcategories of $\rep(Q)$.

The previous theorem can be interpreted for the category coh$(\mathbb{X})$ of coherent sheaves over a weighted projective line $\mathbb{X}$ of positive Euler characteristic (that is, of domestic tame type). The reader is referred to \cite{Len} for an introduction to coherent sheaves over a weighted projective line.
It is well known that when $\mathbb{X}$ is a weighted projective line of tame domestic type, then coh$(\mathbb{X})$ is derived equivalent to $\rep(Q)$ for $Q$ a Euclidean quiver; see \cite[page 126]{Len}. The abelian category  coh$(\mathbb{X})$
is also equivalent to the category of sheaves on a Fano orbifold with coarse
moduli space
$\mathbb{P}^1,$ and to the category of sheaves of modules on a Fano hereditary order
as can be seen, for example, in \cite{ChanSplitting}, \cite{ChanIngalls}.
Here Fano is equivalent to the fact that
$\sum(1- 1/r_i)<2$ where the $r_i$ represent the widths of
the non-homogeneous tubes, the orders of the non-trivial stabilizers of the
orbifold, or the ramification indices of the hereditary order.
 Hence, we can immediately derive the following as a corollary of
the previous theorem.

\begin{Theo} \label{coh}
Let $\mathbb{X}$ be a weighted projective line of tame domestic type.  Then an abelian and extension-closed
subcategory $\mathcal B$ of coh$(\mathbb{X})$ is the subcategory of
$\theta$-semi-stable objects for some $\theta$ if and only if either:
\begin{enumerate}[$(i)$]
\item $\mathcal{B}$ is equivalent to the category of coherent sheaves over a weighted projective line,
\item $\mathcal{B}$ is equivalent to the torsion part of a category of coherent sheaves over a weighted projective line,
\item $\mathcal B$ is the additive hull of two abelian, extension-closed subcategories $\mathcal B_1,\mathcal B_2$ that are orthogonal to each other and where $\mathcal{B}_1$ is as in (i) or (ii) and $\mathcal{B}_2$ is of finite representation type.
\end{enumerate}
\end{Theo}

\section{Canonical decomposition on the regular hyperplane} \label{section6}

In this section, $Q$ is assumed to be a (connected) Euclidean quiver. Let us introduce some notation. Note that a representation is quasi-simple if it is a simple object in the full subcategory of $\rep(Q)$
consisting of the regular representations.  For convenience, a regular real Schur root which corresponds to a quasi-simple representation will be called a \emph{quasi-simple root}. Hence, the null root is not considered to be a quasi-simple root here. We define $H_\delta^{ss} \subseteq H_\delta^{+}$ to be the convex cone in $\mathbb{Q}^n$ generated by the quasi-simple roots and $\delta$. In particular, $H_\delta^{ss} \subseteq (\mathbb{Q}^n)_+$. We call it the \emph{regular cone}
of $Q,$ or of $\mathbb{Q}^n$. Our goal in this section is to describe the
geometry of $H_\delta^{ss}$ and the
structure of the canonical decomposition for $d\in H_\delta^{ss}$. The reader is referred to \cite{Bar} for the basic notions on convex geometry that are used in this section.

We label the non-homogeneous tubes in the Auslander-Reiten quiver of $\rep(Q)$, by $1, \ldots, N$. For $1 \le i \le N$, let $r_i$ be the rank of the tube labeled $i$ and let $\beta_{i1},\dots,\beta_{ir_i}$ be
the quasi-simple roots in that tube.
It is well known (see \cite[Theorem XIII 2.1]{AS2}) that
$$\sum_{i=1}^N (r_i - 1) = n-2.$$
Hence there are exactly $n-2+N$ quasi-simple roots. If $d_1,d_2$ are roots corresponding to the indecomposable representations $M_1,M_2$, respectively, and $M_2$ is the Auslander-Reiten translate of $M_1$, then we write $\tau(d_1) = d_2$ and say that $d_2$ is the \emph{Auslander-Reiten translate} of $d_1$.  If $C$ denotes the Coxeter transformation, this just means that $d_2 = C(d_1)$.

\begin{Lemma}\label{Lemma:dep} The space of linear dependencies among
the vectors $\beta_{ij}$ is $(N-1)$-dimensional.  The linear dependencies
are spanned by those that arise from the fact that the sum of the
quasi-simple roots of any non-homogeneous tube is equal to that of any other
non-homogeneous tube.
\end{Lemma}

\begin{proof} It is well known that each tube contains representations whose dimension vector is the null root and
that these representations have a filtration by quasi-simples in which each quasi-simple in the tube appears
exactly once.  This implies that the sum of  the quasi-simple roots in any
tube equals the null root.  This gives rise to $(N-1)$ linearly independent
dependencies among the vectors $\beta_{ij}$.
We need therefore
only verify that there are no additional dependencies.

Let $\sum c_{ij}\beta_{ij}=0$ for some constants
$c_{ij}$.  Pick some tube with index $i_0$ and some quasi-simple with index
$j_0$.
There are two dimension
vectors of quasi-simples which
have non-zero pairing with $\beta_{i_0j_0}$, namely $\beta_{i_0j_0}$ and its
inverse AR translate.  Specifically,
$$\langle \beta_{i_0j_0},\beta_{i_0j_0}\rangle = 1 \textrm { and }
\langle \tau^{-1}\beta_{i_0j_0},\beta_{i_0j_0}\rangle = -1.$$
Since
$\sum c_{ij}\beta_{ij}=0$, we must have that $\langle \sum c_{ij}\beta_{ij},
\beta_{i_0j_0}\rangle = 0$.  It follows that the coefficients of
$\beta_{i_0j_0}$ and $\tau^{-1}\beta_{i_0j_0}$ must be equal.  By considering all
possible choices of $j_0$ for the tube $i_0$, we see that all the coefficients
corresponding to quasi-simples from tube $i_0$ must be equal.  The same argument applies
to the other tubes, and thus the linear dependence which we started with
lies in the span of those we found initially.
\end{proof}

To describe the facets of $H_\delta^{ss}$, we introduce some notation.  If $|Q_0| > 2$, write $R$ for the set of $N$-tuples $(a_1,\dots,a_N)$ with
$1\leq a_i\leq r_i$.  For $(a_1,\dots,a_N)\in R$, write $F_{(a_1,\dots,a_N)}$ for the convex cone in $(\mathbb{Q}^n)_+$ generated by the
quasi-simple roots except for $\beta_{ia_i}$ for $1\leq i \leq N$.

\begin{Prop}  \label{boundaryofregular}Suppose that $|Q_0|> 2$. The boundary facets of $H_\delta^{ss}$
are exactly the $F_I$
for $I\in R$.
\end{Prop}

\begin{proof}
For $1\leq i \leq N$, write
$P_i$ for the cone generated by the quasi-simples from the $i$-th
tube.  From Lemma \ref{Lemma:dep}, we see that the dimension vectors of the quasi-simple representations in a given tube are linearly independent. Therefore, $P_i$ is a simplicial cone
containing the ray
generated by $\delta$.

Consider the $(n-1)$-dimensional vector space
$\mathbb Q^n/\langle \delta\rangle$, and
let $\overline P_i$ be the subspace of this quotient space generated by $P_i$.
By Lemma \ref{Lemma:dep}, the $\overline P_i$ are complementary.
It follows that codimension
one facets of $H_\delta^{ss}$
are each formed by taking the convex hull of
one facet from each $P_i$.  The description of the facets given in
the statement of the proposition then follows.
\end{proof}

If $\Sigma_1$ and $\Sigma_2$ are two simplicial complexes on disjoint
vertex sets $V_1$, $V_2$, then the \emph{simplicial join} of $\Sigma_1$ and $\Sigma_2$
is defined by saying that a set $F\subseteq V_1\cup V_2$ is a face if and
only if $F\cap V_1$ is a face of $\Sigma_1$ and $F\cap V_2$ is a face of
$\Sigma_2$. In the sequel, we will use the simplicial join of at most three simplices.  Some low dimensional
examples include a square, an octahedron and a triangular bipyramid.

\begin{Cor} Suppose that $|Q_0| > 2$. The combinatorial structure of the boundary of
$H_{\delta}^{ss}$
can be described as a cone over the simplicial join of the boundary of $N$ simplices, one for each non-homogeneous tube,
where the $i$-th simplex has as vertices the quasi-simples from the $i$-th tube.
\end{Cor}

\medskip

\begin{Exam}
Let $Q$ be the quiver below:

$$\begin{tikzpicture}[>=stealth',shorten >=1pt,node distance=1.2cm,on grid,initial/.style    ={}]
  \node[draw=none,fill=none]          (P2)                        {$2$};
  \node[draw=none,fill=none]          (P1) [above =of P2]    {$1$};
  \node[draw=none,fill=none]          (P4) [right =of P2]    {$4$};
  \node[draw=none,fill=none]          (P3) [right =of P1]    {$3$};
\tikzset{mystyle/.style={->}}
\tikzset{every node/.style={draw=none,fill=none}}
\path (P1)     edge [mystyle]       (P2)
      (P2)     edge [mystyle]        (P4)
      (P1)     edge [mystyle]      (P3)
      (P3)     edge [mystyle]      (P4);
\end{tikzpicture}$$
The following is an affine slice of
$(\mathbb{Q}^n)_+$, with $H_\delta^{ss}$ indicated:

$$
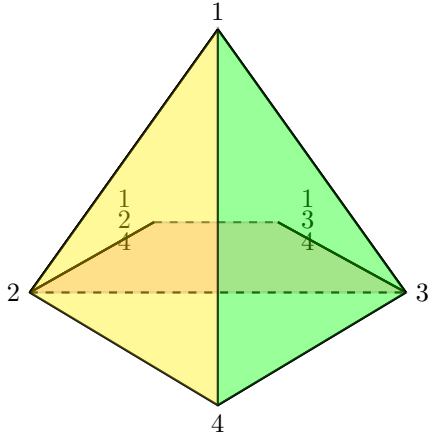
\begin{tikzpicture}[thick,scale=5]

\coordinate[label=above:$1$] (A1) at (0.5,0.5);
\coordinate[label=left:$2$] (A2) at (0,-0.2);
\coordinate[label=right:$3$] (A3) at (1,-0.2);
\coordinate[label=below:$4$] (A4) at (0.5,-0.5);
\coordinate[label=right:$\renewcommand\arraystretch{0.5}\begin{array}{c}1\\3\\4\end{array}$]  (B3) at (0.66,-0.013); % average of A1 A3 A4
\coordinate[label=left:$\renewcommand\arraystretch{0.5}\begin{array}{c}1\\2\\4\end{array}$] (B2) at (0.33,-0.013);% average of A1 A2 A4

\begin{scope}[thick,dashed,,opacity=0.6]
\draw (A2) -- (A3);
\draw (B2) -- (B3);
\end{scope}

\begin{scope}[thick,,,opacity=0.6]
\draw (A2) -- (B2);
\draw (A3) -- (B3);
\draw (A2) -- (A1);
\draw (A1) -- (A3);
\draw (A3) -- (A4);
\draw (A4) -- (A2);
\draw (A1) -- (A4);
\end{scope}

\draw[fill=yellow,opacity=0.4]  (A1) -- (A2) -- (A4) --cycle;
\draw[fill=green,opacity=0.4]  (A1) -- (A4) -- (A3) --cycle;
\draw[fill=red,opacity=0.1] (B3) -- (B2) -- (A2) -- (A3) --cycle;

\end{tikzpicture}$$

Observe that $H_\delta^{ss}$ is not equal to
the intersection of the affine span of $H_\delta$ with $(\mathbb{Q}^n)_+$ in this case
(contrary to what one might have expected from smaller examples). For instance, the dimension vector $d_{S_1} + d_{S_4}$ lies in $H_\delta$ but not in  $H_\delta^{ss}$.

\end{Exam}

\begin{Cor} All the regular real Schur roots lie on the boundary of
$H_\delta^{ss}$. \end{Cor}
\begin{proof} Using the fact that the stable tubes are standard (see Proposition \ref{ARShape}) and that for $X$ regular, $\Ext^1(X,X) \cong D\Hom(X,\tau X)$, we get that the regular real Schur roots correspond
to representations which admit a filtration by quasi-simples from some
tube which does not include all the quasi-simples from that tube.
It therefore follows that each real Schur root lie on the boundary of
some $P_i$ (defined in the proof of Proposition~\ref{boundaryofregular}), and thus on the boundary of $H_\delta^{ss}$.
\end{proof}

For $Q$ the Kronecker quiver, that is, when $|Q_0|=2$, we define $R$ to be the set with one element $\emptyset$ and we set $F_\emptyset = \emptyset$.

For each facet $F_{I}$ of $H_\delta^{ss}$, denote by $C_{I}$ the
simplicial cone in $(\mathbb{Q}^n)_+$ generated by $F_{I}$ and
$\delta$.
Note that in the Kronecker case, there is a unique cone $C_\emptyset$, and it is equal to the ray generated by $\delta$.

\begin{Prop}
Let $d$ be a dimension vector lying in $H_\delta^{ss}$.
\begin{enumerate}[$(1)$]
    \item If $d$ lies on a facet $F_I$ of $H^{ss}_\delta$, then the canonical decomposition of $d$ only involves regular real Schur roots in $F_I$.
\item If $d$ lies in the relative interior of $H_\delta^{ss}$, say in $C_I$, then the canonical decomposition of $d$ involves the null root and possibly some regular real Schur roots in $F_I$. \end{enumerate}
\end{Prop}
\begin{proof}

From what we just proved, we have that $\delta$ lies in the relative interior of $H_\delta^{ss}$.  Moreover, all the quasi-simple roots lie on the boundary of $H_\delta^{ss}$.

Suppose first that $d$ is a dimension vector lying on the facet $F_I$ of $H_\delta^{ss}$.
The  additive extension-closed abelian subcategory of $\rep(Q)$ generated by the quasi-simple representations corresponding to the roots in $F_I$ is equivalent to the category of representations of a quiver which is a union of quivers of type $\mathbb{A}$. This observation, together with Proposition \ref{Kac1}, gives that the canonical decomposition of $d$ only involves real Schur roots lying on $F_I$.

Suppose now that $d$ lies in the relative interior of $H_\delta^{ss}$, say in $C_I$.
We can decompose $d$ as $d = d_1 + r\delta$ where $d_1 \in F_I$ and $r \in \mathbb{Z}$.  Let $d_1 = d(1) \oplus \cdots \oplus d(m)$ be the canonical decomposition of $d_1$ which, by the above argument, only involves real Schur roots in $F_I$.  By Proposition \ref{Kac1}, it is clear that
$$d = d(1) \oplus \cdots \oplus d(m) \oplus \delta^r$$
is the canonical decomposition of $d$ since ${\rm ext}(d(i),\delta) = {\rm ext}(\delta,d(i))=0$ for all $i$.
\end{proof}

\section{Ss-equivalence and proof of Theorem \ref{ssequiv}}

In this section, we suppose that $Q$ is a Dynkin or Euclidean quiver with $n$ vertices, unless otherwise indicated.  We prove Theorem \ref{ssequiv}, which describes when two
dimension vectors determine the same semi-stable subcategories.

Recall that when $Q$ is Euclidean, the quadratic form $q(x) = \langle x, x \rangle$ is positive semi-definite and its radical is of rank one, generated by the null root $\delta$.
Recall also that the convex set $H_\delta^{ss}$ is defined as the cone generated by the quasi-simple roots, if $Q$ is of Euclidean type and contains more than $2$ vertices,
and $H_\delta^{ss}=\{[\delta]\}$, if $Q$ is the Kronecker quiver.

\medskip

Let $X$ be an exceptional representation in $\rep(Q)$ and from Proposition \ref{GL}, let $Q_X$ be the quiver with $n-1$ vertices for which $^\perp  X \cong \rep(Q_X)$. Note that $Q_X$ is a possibly disconnected Euclidean or Dynkin quiver.
Let $$F_{d_X} : {}^\perp X \to  \rep(Q_X)$$
be an exact functor
which is an equivalence.  Denote by $G_{d_X}$ a quasi-inverse functor.
Let $S_1^X,S_2^X,\ldots,S_{n-1}^X$ be the non-isomorphic simple objects in $^\perp X$,
called the \emph{relative simples} in $\p X$. It is clear that
$K_0(\p X)$, which is isomorphic to $K_0(\rep(Q_X))$,
is the subgroup of $K_0(\rep(Q))$
generated by the classes $[S_1^X],\ldots,[S_{n-1}^X]$ in $K_0(\rep(Q))$. Put
$$\varphi_X: K_0(\p X) \to K_0(\rep(Q_X))$$
for the canonical isomorphism. Let $\langle \;,\; \rangle_X$ be the Euler form for
$\rep(Q_X)$.  Since $F_{d_X}$ is exact, $\varphi_X,\varphi_X^{-1}$ are isometries, that is, for $a,b \in K_0(\rep(Q_X))$, we have $$\langle a,b \rangle_X = \langle \varphi_X^{-1}(a),\varphi_X^{-1}(b) \rangle.$$

Let us consider $P(X)$ the (possibly empty) set of all dimension vectors of the indecomposable projective representations of $\rep(Q)$ in $^\perp X$. We let $H_{d_X}^{ss}$ denote the cone in $\mathbb{Q}^n$ generated by the elements of $-P(X) \cup \{d_{S_1^X}, \ldots, d_{S_{n-1}^X}\}$. As a first property of $H_{d_X}^{ss}$, we have the following.

\begin{Lemma} \label{lemmad+d-} Let $\alpha$ be a real Schur root and let $d \in \Z^n$ with $d \in H_\alpha^{ss}$. Then $d_+, d_- \in H_\alpha^{ss}$.
\end{Lemma}

\begin{proof} Let $d_1, \ldots, d_r$ be the dimension vectors of the relative simple modules of $\p M(\alpha)$. If $M(\alpha)$ is sincere, $d=d_+$, $d_-=0$ and there is nothing to prove. Otherwise, let $q_1, \ldots, q_s$ be the dimension vectors of the indecomposable projective representations that are in $\p M(\alpha)$. Since $d \in H_\alpha^{ss}$, $d = \sum_{i=1}^ra_id_i - \sum_{j=1}^sb_jq_j$ where the $a_i, b_j$ are non-negative integers. Now, consider the projective representation $P_0$ of dimension vector $\sum_{j=1}^sb_jq_j$ and consider $M_0 = \bigoplus_{i=1}^rM(d_i)^{a_i}$, which is a representation in $\p M(\alpha)$ of dimension vector $\sum_{i=1}^ra_id_i$. If $\Hom(P_0,M_0)=0$, then $d_+ = d_{M_0}, d_- = -d_{P_0}$ and we are done. Otherwise, let $f: P_0 \to M_0$ be a non-zero homomorphism with kernel $P_1$ and cokernel $M_1$. Since $P_0,M_0 \in \p M(\alpha)$ and $\p M(\alpha)$ is an abelian subcategory of $\rep(Q)$, we have $P_1,M_1 \in \p M(\alpha)$. Since $\rep(Q)$ is hereditary, $P_1$ is projective. Observe that $d = d_{M_1} - d_{P_1}$. If $\Hom(P_1, M_1)=0$, then $d_+ = d_{M_1}, d_- = -d_{P_1}$ and we are done. Otherwise, we let the process goes by induction and we get a sequence of submodules $\cdots \subset P_1 \subset P_0$, a sequence of quotients $M_0 \to M_1 \to \cdots$ such that $\Hom(P_i,M_i)\ne 0$ if and only if both $P_{i+1} \subset P_i$ and $M_i \to M_{i+1}$ are proper. Moreover, $d = d_{M_i} - d_{P_i}$ for all $i \ge 0$. Since $P_0$ is finite dimensional, there exists $t \ge 1$ such that $P_t = P_j$ whenever $j \ge t$. This gives $\Hom(P_t, M_t)=0$. Thus, $d_+ = d_{M_t}, d_- = -d_{P_t}$. By construction, all $P_i$ and $M_j$ are in $\p M(\alpha)$, which proves the lemma.
\end{proof}

Now, if $Q$ is Euclidean and $\delta \in H_{d_X}^{ss}$, then
$\varphi_X(\delta)$ is clearly the null root for the quiver $Q_X$ of $\p X$. Observe that even when $Q$ is connected, $Q_X$ may be disconnected (a union of quivers of Dynkin type and at most one quiver of Euclidean type), thus, $\varphi_X(\delta)$ may be non-sincere.

Observe also that it is possible to have a dimension vector $f \in H^{ss}_{d_X}$ and a representation $M \in \rep(Q,f)$ with $M$ not isomorphic to any representation in
$\p X$.  However, this does not happen for general representations.
The notion of \emph{general representation} used below was defined
after Proposition~\ref{Kac1}.

\begin{Lemma} \label{Lemma2}
Let $X$ be an exceptional representation and $f \in H^{ss}_{d_X}$ a dimension vector. Then there is a general representation of dimension vector $f$ in $\p X$.
Moreover, $M$ is a general representation of dimension vector $\varphi_X(f)$ in $\rep(Q_X)$
if and only if $G_{d_X}(M)$ is a general representation of dimension vector
$f$ in $\p X$. In particular, $\varphi_X, \varphi_X^{-1}$ preserve the canonical decomposition.
\end{Lemma}

\begin{proof}
Since $f\in H^{ss}_{d_X}$, there exists a representation of dimension vector $f$ in $\p X$.  Let $\mathcal{U}$ be a non-empty open set of $\rep(Q,f)$ as in the definition of the canonical decomposition of $f$. A representation $N$ in $\rep(Q,f)$ lies in $\p X$ if and only if the determinant $C^N(X)$ does not vanish.  This defines a non-empty open set $\mathcal{U}'$ in $\rep(Q,f)$.  Being an affine space, $\rep(Q,f)$ is irreducible. The first statement follows from the fact the the intersection $\mathcal{U} \cap \mathcal{U'}$ is non-empty and $\mathcal{U}$ only contains general representations. Being exact equivalences, it is clear that $F_{d_X}, G_{d_X}$ preserve the general representations.
\end{proof}

If
$X$ is exceptional, then the cd-fan for $Q$ restricts to a simplicial fan in $H_{d_X}^{ss} \cap (\mathbb{Q}^n)_+$. From Lemma
\ref{Lemma2}, this simplicial fan on $H_{d_X}^{ss}\cap (\mathbb{Q}^n)_+$ corresponds, under $\varphi_X$, to the cd-fan for $\rep(Q_X)$. Let us note that it is not true that the cp-fan for $Q$ restricts to the cp-fan for $\rep(Q_X)$. The reason is that the relative projective objects in $\rep(Q_X)$ are not necessarily projective objects in $\rep(Q)$.

Recall, from Section \ref{section6}, the definition of the set $R$ when $Q$ is Euclidean. To unify the notations, we set $R = \emptyset$ when $Q$ is Dynkin. For $Q$ Dynkin or Euclidean, let $\hyp=\hyp_Q = \{C_I\}_{I \in R} \cup \{H_\alpha^{ss}\}_{\alpha}$ where $\alpha$ runs through the set of all real Schur roots.
Let $\lat=\lat_Q$ be the set of all subsets of $\mathbb{Q}^n$ that can be expressed
as intersections of some subset of $\hyp$.  This is called the
\emph{intersection lattice} of $\hyp$.
For $L\in \lat$, define the faces of $L$ to be
the connected components of the set of points which are in $L$ but
not in any smaller intersection.  Define the faces of $\lat$ to be
the collection of all faces of all the elements of $\lat$. Given $d \in \Z^n$, $\hyp_d \subseteq \hyp$ consists of all the elements in $\hyp$ containing $d$.

\begin{Lemma}\label{reals} Let $Q$ be a Dynkin or Euclidean quiver, $\alpha$ a
real Schur root and $d \in \Z^n$.  Then $d\in H_\alpha^{ss}$ if and only if $M(\alpha)\in \rep(Q)_d$.
\end{Lemma}

\begin{proof} Suppose that $M(\alpha)\in \rep(Q)_d$. Then from Corollary \ref{cor1}, there exist a positive integer $m$ and a representation
$V$ supported over $Q_{\alpha}$ with $\langle d_V, - \rangle \sim_\alpha \langle md, - \rangle$ and with $C^{V}(M(\alpha)) \ne 0$. Then $V \in {}^\perp M(\alpha)$ and $d_V \in H_\alpha^{ss}$. If $M(\alpha)$ is sincere, then $d_V = md$ and hence $d \in H_\alpha^{ss}$. Otherwise, $\langle d_V-dm, - \rangle \sim_\alpha 0$. There exists $f \in \Z^n$ having a support disjoint from $\alpha$ such that $\langle d_V-dm, - \rangle = f$, which gives $d_V-dm = fE^{-1}$, where $E$ is the matrix of the Euler form. But $fE^{-1}$ lies in the span of the dimension vectors of the projective representations whose tops are not supported over $Q_{\alpha}$. This gives $d\in H_\alpha^{ss}$. Conversely, assume
$d \in H_\alpha^{ss}$. Write $d = d_+ + d_-$, where $d_+$ is a dimension vector and $-d_-$ is the dimension vector of a projective representation such that $\langle d_-, d_+ \rangle = 0$. By Lemma \ref{lemmad+d-}, we have $d_+, d_- \in H_\alpha^{ss}$. Then a general representation of $Q_{M(\alpha)}$ of
dimension vector $\varphi_{M(\alpha)}(d_+)$ is a general representation of $Q$ of dimension
vector $d_+$, and this gives us a supply of representations of
dimension vector $d_+$ which lie in $\p M(\alpha)$.  Any such representation $V$
will satisfy $C^V(M(\alpha))\ne 0$. Since $\langle d, - \rangle$ and $\langle d_+, - \rangle$ agree on the support of $M(\alpha)$, we see by Propositions \ref{King} and \ref{dnonreg} that $M(\alpha)\in \rep(Q)_d$.
\end{proof}

We will now state the main theorem of this section, and prove it in the
Dynkin case.  The proof in the Euclidean case is similar but
requires some further lemmas,
so we defer it to the end of this section.

\begin{Theo}\label{ssequiv} Let $Q$ be a Dynkin or Euclidean quiver and $d_1,d_2 \in \Z^n$.
Then $d_1$ and $d_2$ are ss-equivalent if and only if $\hyp_{d_1}=\hyp_{d_2}$.
\end{Theo}

\begin{proof}[Proof of Theorem \ref{ssequiv} in the Dynkin case]
Let $Q$ be a Dynkin quiver.  Suppose that $d_1$ and $d_2$ are ss-equivalent.
By definition, $\rep(Q)_{d_1}=\rep(Q)_{d_2}$, and Lemma \ref{reals}
characterizes $\hyp_{d_1}$ and $\hyp_{d_2}$ in terms of this subcategory,
so they must be equal.  (Since $Q$ is Dynkin, $\hyp$ consists only of
cones of the form $H_{\alpha}^{ss}$.)

Conversely, suppose $\hyp_{d_1}=\hyp_{d_2}$.  By Lemma \ref{reals},
$\rep(Q)_{d_1}$ and $\rep(Q)_{d_2}$ contain the same exceptional
representations.  Since $Q$ is Dynkin, all indecomposables are exceptional,
and thus $\rep(Q)_{d_1}=\rep(Q)_{d_2}$.
\end{proof}

Suppose now that $Q$ is Euclidean.
For $I \in R$, we denote by $\mathcal{W}_I$ the abelian, extension-closed subcategory generated by the indecomposable representations of dimension vectors in $F_I$.
There is a unique quasi-simple in each non-homogeneous tube
not contained in $\mathcal W_I$.

An indecomposable representation lying in a non-homogeneous tube and
having dimension vector the null root will be called a
\emph{singular-isotropic representation}.
For each quasi-simple of a non-homogeneous tube, there is a
unique corresponding singular-isotropic representation for which there is a monomorphism from the quasi-simple to the singular-isotropic representation.

For $Q$ having more than $2$ vertices and $I\in R$, define $Z_I$ to be the direct sum of the
singular-isotropic representations corresponding to the quasi-simples
not in $\mathcal W_I$.  We have the following simple lemma:

\begin{Lemma}\label{pZ}
For $Q$ Euclidean and having more than $2$ vertices, $\p Z_I$ consists of the additive hull of $\mathcal W_I$ together with the
homogeneous tubes.
\end{Lemma}

\begin{proof} Lemma \ref{Lemma1} tells us that $\p Z_I$ is contained in
$\mathcal{R}$eg.  It is clear that the homogeneous tubes lie in
$\p Z_I$.  The rest is just a simple check within each of the
non-homogeneous tubes.
\end{proof}

We can now prove an analogue of Lemma \ref{reals} for the $C_I$.

\begin{Lemma} \label{siso} Let $Q$ be a Euclidean quiver having $n$ vertices and let $d \in \Z^n$.  If $n > 2$, then for $I\in R$, we have
$d \in C_I$ if and only if $Z_I \in \rep(Q)_d$.  If $n=2$, then $d \in C_I = H_\delta^{ss}$ if and only if $\rep(Q)_d = \mathcal{R}$eg. \end{Lemma}

\begin{proof} The case $n=2$ is trivial. If $Z_I\in \rep(Q)_d$, since it is sincere, then there is some representation
$V$ with dimension vector $md$ such that $C^V(Z_I)\ne 0$, and indeed,
a Zariski open subset of the representations with dimension $md$ have
this property.
By Lemma
\ref{pZ}, this implies that the canonical decomposition of
$md$ must contain Schur roots from $\mathcal W_I$
together with some multiple of $\delta$, and this implies that $d$ lies
in $C_I$.

Conversely, if $d$ lies in $C_I$, then the canonical decomposition of
$d$ consists of Schur roots from $\mathcal W_I$,
together with some non-negative multiple of $\delta$.  Lemma \ref{pZ}
implies that there exist representations of dimension vector $d$ which
lie in $\p Z_I$, so $Z_I\in \rep(Q)_d$.
\end{proof}

Lemmas \ref{reals} and \ref{siso} tell us that $\hyp_d$ contains enough
information to determine
exactly which exceptional or singular-isotropic indecomposable representations lie in $\rep(Q)_d$.  We will
use the
following lemma to show that this is enough information
to reconstruct $\rep(Q)_d$.

\begin{Lemma} \label{intersections}
Let $Q$ be Euclidean and let $\mathcal A$, $\mathcal B$ be
extension-closed abelian subcategories
of $\rep(Q)$.  Suppose $\mathcal A$ and $\mathcal B$ have the same
intersection with the exceptional representations and the
indecomposables whose dimension vector is the null root.  Then
$\mathcal A=\mathcal B$.
\end{Lemma}

\begin{proof}  Suppose there is some object $X\in \mathcal A$.  We want
to show that $X\in \mathcal B$.  We may assume that $X$ is indecomposable.
If $X$ is exceptional or $X$ is at the bottom of a homogeneous tube, then $X\in \mathcal B$ by assumption, so assume
otherwise.

Suppose that the
dimension vector of $X$ is $m\delta$ for some $m>1$.  We claim that
$X$ admits a filtration by representations each of whose dimension vectors
is the null root, and each of which lies in $\mathcal A$.  The proof
is by induction on $m$.  Suppose that the tube containing $X$ has width $r$. There is a path of length $(m-1)r$ consisting of irreducible epimorphisms from $X$ to a singular-isotropic representation $Y$. There is also a path of length $(m-1)r$ consisting of irreducible monomorphisms from $X$ to a representation $Z$ of dimension vector $(2m-1)\delta$. This yields a monomorphism $X \to Y \oplus Z$ whose cokernel is clearly $X$.
Hence, there
is a short exact sequence
$$ 0 \rightarrow X \rightarrow Y\oplus Z \rightarrow X \rightarrow 0.$$  It
follows that $Y \in \mathcal A$.
Now, let $K$ be the kernel of the surjection from $X$ onto $Y$.  Since
$\mathcal A$ is abelian, $K$ also lies in $\mathcal A$, and its dimension
vector is $(m-1)\delta$.  By induction, it admits a filtration as desired,
and therefore so does $X$.
Now since each of the terms of the filtration lies in
$\mathcal A$ and has dimension vector $\delta$, they each also lie in
$\mathcal B$, and therefore $X\in \mathcal B$.

Finally, suppose that the dimension vector of $X$ is of the form
$m\delta+\alpha$, where $\alpha$ is a real Schur root.  Similarly to the
previous situation, there is an extension of $X$ by $X$ which has
$M(\alpha)$ as an indecomposable summand, so $M(\alpha)$ lies in
$\mathcal A$, and so does the kernel $K$ of the map from $X$ to $M(\alpha)$.
The dimension vector of $K$ is $m\delta$, so by what we have already established, it lies in $\mathcal B$, and so does $M(\alpha)$, so $X$ does as well.
\end{proof}

\begin{proof}[Proof of Theorem \ref{ssequiv} in the Euclidean case]
If $d_1$ and $d_2$ are ss-equivalent, then, by definition,
$\rep(Q)_{d_1}=\rep(Q)_{d_2}$, and now Lemmas \ref{reals} and \ref{siso}
characterize $\hyp_{d_1}$ and $\hyp_{d_2}$ in terms of this subcategory,
so they are equal.

Conversely, if $\hyp_{d_1} = \hyp_{d_2}$, then by Lemmas \ref{reals} and \ref{siso}, $\rep(Q)_{d_1}$ and $\rep(Q)_{d_2}$ agree as
to their intersection with exceptional representations and singular-isotropic
representations.  In order to apply Lemma \ref{intersections}, we also
need to check their intersections with the quasi-simples of homogeneous tubes.
If $\hyp_{d_1}$ and $\hyp_{d_2}$ do not include any
$C_I$, then neither $d_i$ lies on $H_\delta^{ss}$, so neither
$\rep(Q)_{d_1}$ nor $\rep(Q)_{d_2}$ contains any homogeneous tubes by Lemma \ref{Lemma1} and Proposition \ref{dnonreg}.
On the other hand, if $\hyp_{d_1}$ and $\hyp_{d_2}$ do contain some
$C_I$, then $d_1$ and $d_2$ both lie on the regular hyperplane, and thus
both $\rep(Q)_{d_1}$ and $\rep(Q)_{d_2}$ contain all the homogeneous tubes by Lemmas \ref{Lemma1} and \ref{onlyreg} and Proposition \ref{dreg}.
In either case, we see that $\rep(Q)_{d_1}$ and $\rep(Q)_{d_2}$ agree as
to their intersections with quasi-simples from homogeneous tubes as well,
and therefore Lemma \ref{intersections} applies to tell us that
$\rep(Q)_{d_1}=\rep(Q)_{d_2}$.  \end{proof}

\section{Properties of the cp-fan and proof of Theorem \ref{two}}

In this section, we prove Theorem \ref{two}, describing when two
elements of $\Z^n$ have canonical presentations whose summands correspond to the same
rays in the cp-fan. Recall that since $Q$ is Dynkin or Euclidean, the rays of the cp-fan correspond bijectively to the summands occurring in cp-presentations. All the quivers in this section are Dynkin or Euclidean.

In the sequel, an additive $k$-category $\mathcal{C}$ will be called \emph{representation-finite}, or of \emph{finite type}, if it contains finitely many indecomposable objects, up to isomorphism.
In the Euclidean case, all the categories $\mathcal{W}_I$ are representation-finite, abelian, extension-closed subcategories of $\mathcal{R}$eg.
The following shows that the $\mathcal{W}_I$ are the maximal such subcategories.

\begin{Lemma} \label{Lemma3}
Let $Q$ be Euclidean and $\mathcal{C}$ be an abelian, extension-closed subcategory of $\rep(Q)$ contained in $\mathcal{R}$eg.  Then $\mathcal{C}$ is representation-finite if and only if $\mathcal{C}$ is contained in some $\mathcal{W}_I$.
\end{Lemma}

\begin{proof} We only need to prove the necessity.
Suppose that $\C$ is representation-finite but not included in any of the $\mathcal{W}_I$.
If $\C$ contains an indecomposable non-exceptional representation $X$, then $\C$ clearly contains infinitely many indecomposable representations lying in the same component as $X$ of the Auslander-Reiten quiver of $\rep(Q)$, so
suppose this does not hold.
Since $\mathcal C$ is not contained in any $\mathcal W_I$, there exists a non-homogeneous tube $\mathcal{T}$, say of rank $r$, for which $\mathcal{C} \cap \mathcal{T}$ is not contained in any abelian, extension-closed subcategory of $\mathcal{T}$ generated by $r-1$ quasi-simple representations. Let $X$ be an exceptional representation in $\mathcal{C} \cap \mathcal{T}$ of largest quasi-length.  Then $X$ is such that $\Ext^1(X,\mathcal{C} \cap \mathcal{T}) = \Ext^1(\mathcal{C} \cap \mathcal{T},X)=0$ since otherwise, this would provide a representation in $\mathcal{C} \cap \mathcal{T}$ with quasi-length larger than that of $X$. Suppose that the quasi-simple composition factors of $X$ are $S, \tau S, \ldots, \tau^mS$, $0 \le m <r$.  Since $X$ is exceptional, $\tau^{m+1}S \not \cong S$ and $\tau^{m+1}S$ is not a quasi-simple composition factor of $X$ and any $Y \in \mathcal{C} \cap \mathcal{T}$ having $\tau^{m+1}S$ as a quasi-simple composition factor is such that $\Ext(X,Y)\cong D\Hom(Y,\tau X) \ne 0$, which is impossible. Hence, $\mathcal{C} \cap \mathcal{T}$ is contained in the extension-closed abelian subcategory of $\mathcal{T}$ generated by all the quasi-simple representations except $\tau^{m+1}S$. This is a contradiction.
\end{proof}

We say that two elements $d_1,d_2$ in $\Z^n$ are \emph{cp-equivalent} if the canonical presentations of $d_1,d_2$ involve the same summands, up to multiplicity.

\begin{Exam} Let $Q$ be the quiver $2 \longrightarrow 1$. The indecomposable representations, up to isomorphism, are the simple representations $S_1, S_2$ at vertices $1,2$, respectively and the projective-injective representation $P_2$. The canonical presentation of $(1,2)$ is $(1,1)+(0,1)$ and that of $(2,3)$ is $2(1,1) + (0,1)$. Therefore, $(1,2), (2,3)$ are cp-equivalent. The canonical presentation of $(-2,-3)$ is $3(-1,-1) + 1(1,0)$. So $(-2,-3)$ is cp-equivalent to $(-b,-a)$ for $a,b$ positive integers with $b < a$.
\end{Exam}

The following result gives an explicit description of when two elements are cp-equivalent.

\begin{Theo}\label{two}
Let $Q$ be a Dynkin or Euclidean quiver.  Two elements $d_1, d_2$ in $\Z^n$ are cp-equivalent if and only if
they lie in the same face of $\lat_Q$.
\end{Theo}

\begin{proof}
Let $d_1, d_2 \in \Z^n$. Suppose that $d_1,d_2$ are cp-equivalent. Then any non-negative integer linear
combination of $d_1$ and $d_2$ will also be cp-equivalent to both of them,
by Lemma \ref{IOTW}.
If $d_1,d_2$ do not lie on the same face of $\lat_Q$, either $\{J \in \hyp_Q \mid d_1 \in J\} \ne \{J \in \hyp_Q \mid d_2 \in J\}$ or $\{J \in \hyp_Q \mid d_1 \in J\} = \{J \in \hyp_Q \mid d_2 \in J\}$ but there exists $H \in \hyp_Q$ such that $(\cap_{d_2 \in J}J) \backslash H$ consists of multiple
components, with $d_1$ and $d_2$ in different components.
In both cases, there is a vector $d_3$ which lies on the straight line joining $d_1,d_2$ and $H \in \hyp_Q$ such that $d_3 \in H$ and at least one of $d_1$, $d_2$ does not
lie on $H$.  Without loss of generality, suppose that $d_1 \not\in H$.
As we have already remarked, $d_3$ must be cp-equivalent to $d_1$, and $(d_3)_+$ must be cp-equivalent to $(d_1)_+$.  Also, since $d_3$
lies on $H$, by Lemma \ref{lemmad+d-}, $(d_3)_+$ does.

If $H = H_{\alpha}^{ss}$ for a real Schur root $\alpha$, then there is a general representation $N$ of dimension vector $(d_3)_+$ such that $N \in \p M(\alpha)$ by Lemma \ref{reals}.  But there exists a general representation $N'$ (built from the indecomposable direct summands of $N$) of dimension vector $(d_1)_+$ that has the
indecomposable summands of the same dimensions (though
different multiplicities),
so $N' \in \p M(\alpha)$, showing that $(d_1)_+ \in H_{\alpha}^{ss}$ by Lemma \ref{reals}. Now, the fact that $(d_3)_- \in H_{\alpha}^{ss}$ clearly implies that $(d_1)_- \in H_{\alpha}^{ss}$. This gives, by the convexity of $H_{\alpha}^{ss}$, that $d_1 \in H_{\alpha}^{ss}$, a contradiction.

Similarly, if $|Q_0|>2$ and $H=C_I$ for some
$I\in R$, then $d_3$ is a dimension vector and we deduce that a general representation $N$ of dimension
vector $d_3$ lies in $\p Z_I$ by Lemma \ref{siso}.
As argued above, there will be a general representation $N'$ of dimension vector $d_1$ which
lies in $\p Z_I$, showing that $d_1\in C_I$, a contradiction. If $|Q_0|=2$ and $H = C_\emptyset = H_\delta^{ss}$, then by cp-equivalence, both $d_1,d_3$ are multiples of $\delta$ and it is clear that $d_1 \in H$, a contradiction.
This establishes one direction.

Suppose now that $d_1, d_2$ lie in the same face of $\lat_Q$.
Suppose further that the canonical presentation of $d_1$ involves distinct
vectors $f_1,\dots,f_r$, where $f_i$ is either a Schur root or the negative of the dimension vector of a projective indecomposable representation.  By Lemma \ref{IOTW},
any positive linear combination of
$f_1,\dots,f_r$ will have the same summands appearing in its
canonical presentation as for $d_1$.  Thus, we will be done if we can
show that the fact that $d_2$ lies in the same face of $\lat_Q$, implies
that $d_2$ lies in the cone generated by the $f_i$.
Since the $f_i$ are linearly independent, the boundary facets of this
cone are spanned by any $r-1$ of the $f_i$.  We will show that for any $1\leq i \leq r$,
we have that $\{f_1,\ldots,f_{i-1},f_{i+1},\ldots, f_r\}$ lies on some
$H\in \hyp$ with $f_i \not\in H$.
The rest of the proof is devoted to proving the claim, where it suffices to consider one particular value of $i$. So let us consider
$i=r$.

Suppose first that $\{f_1,\dots,f_r\}$ does not contain the null root.
For each $i$, let $f_i'=f_i$ if $f_i$ is a Schur root and $f_i' = -f_i$ otherwise.
The representations $M(f_i')$ can be ordered into an exceptional sequence:
those corresponding to positive roots can be so ordered because they form
a partial tilting object, and then those corresponding to negative roots
can be put at the front.
Now remove $M(f_r')$.
The resulting exceptional sequence can be completed to a full
exceptional sequence by adding some terms $X_{r},\dots,X_{n}$ to the
end.  If we re-insert $M(f_r')$ into the sequence at the place where it
was before, the resulting sequence is too long to be exceptional, so
there is some $X_k$ with $r\leq k \leq n$ such that
$M(f_r')\not\in \p X_k$, while $M(f_j')\in \p X_k$ for $1 \le j < r$, by virtue of the
exceptional sequence property.  Therefore, $f_r' \not \in H^{ss}_{d_{X_k}}$ and $f_j' \in H^{ss}_{d_{X_k}}$ for $1 \le j < r$. It is clear that for $1 \le j \le r$, $f_j \in H^{ss}_{d_{X_k}}$ if and only if $f_j' \in H^{ss}_{d_{X_k}}$. This shows that $H^{ss}_{d_{X_k}}$ is a hyperplane
of the desired type. This case has now been dealt with, which
completes the proof of the second direction in the Dynkin setting.

Next, assume $Q$ is Euclidean and suppose that $f_r=\delta$.  In this case, all the other $f_i$ are dimension vectors. Moreover $\bigoplus_{1 \le j < r} M(f_j)$ forms a
regular partial tilting object by Lemma \ref{Lemma1}, and its summands can be ordered into an exceptional
sequence.  Complete this to an exceptional sequence by adding some
terms $X_r,\dots,X_n$.  Since a full exceptional sequence cannot consist
entirely of regular objects, there is some $X_k$ which is not regular.
Now $H^{ss}_{d_{X_k}}$ contains $f_1,\dots,f_{r-1}$, but by Lemma \ref{Lemma1} it
does not contain $f_r=\delta$.  Thus $H^{ss}_{d_{X_k}}$ has the desired properties.

Finally, suppose that some $f_j=\delta$ for some $1\leq j\leq r-1$.
For convenience, let $f_{r-1}=\delta$.  Then $M(f_1),\dots,M(f_{r-2})$ and
$M(f_r)$
are
the summands of a partial tilting object and can be ordered into
an exceptional sequence which is contained in the regular component.
Since they can be ordered into an
exceptional sequence contained in
the regular component, they generate an abelian subcategory of finite
representation type, which is therefore contained in some wing
$\mathcal W_I$ by Lemma \ref{Lemma3}.  Delete $M(f_r)$ from the sequence, and then
complete it to an exceptional
sequence which is full inside $\mathcal W_I$, by adding some
terms $X_{r-1},\dots,X_{n-2}$.  Since $M(f_r)$ cannot be added back into
the sequence, we see that there is some $X_k$ such that $M(f_r)$ is not
in $\p X_k$, so $f_r \not\in H_{d_{X_k}}^{ss}$, while by the exceptional
sequence property, $f_t \in H_{d_{X_k}}^{ss}$ for $1\leq t\leq r-2$, and
$f_{r-1} \in H_{d_{X_{k}}}^{ss}$ because $f_{r-1}=\delta$ and $X_k$ is regular.
\end{proof}

\section{Thick subcategories} \label{sectionthick}

This section is devoted to proving some facts concerning extension-closed abelian subcategories of $\rep(Q)$. For part of the section, we assume only
that $Q$ is acyclic; later, we assume that $Q$ is Euclidean.

We begin in an even more general setting. Let $\mathcal{H}$ be a hereditary abelian $k$-category. A (full) subcategory $\mathcal{A}$ of $\mathcal{H}$ is \emph{thick} if it is closed under direct summands and whenever we have a short exact sequence with two terms in $\mathcal{A}$, then the third term also lies in $\mathcal{A}$. We start with the following result which is well known; see for example \cite[Theorem 3.3.1]{Dichev}. We include a proof for the convenience of the reader.

\begin{Prop} \label{DieterVossiek} Let $\mathcal{H}$ be a hereditary abelian $k$-category with a full subcategory $\mathcal{A}$.  Then $\mathcal{A}$ is extension-closed abelian if and only if it is thick.
\end{Prop}

\begin{proof} The necessity follows trivially. Suppose that $\mathcal{A}$ is thick. We need to show that $\mathcal{A}$ has kernels and cokernels.  Let $f: A \to B$ be a morphism in $\mathcal{A}$ with kernel $u: K \to A$, cokernel $v: B \to C$ and coimage $g: A \to E$. Since $\mathcal{H}$ is hereditary and we have a monomorphism $g': E \cong {\rm im}(f) \to B$, we have a surjective map $\Ext^1(g',K): \Ext^1(B,K) \to \Ext^1(E,K)$.
The short exact sequence $0 \to K \to A \to E \to 0$ is an element in $\Ext^1(E,K)$ and hence is the image of an element in
$\Ext^1(B,K)$.  We have a pullback diagram
$$\xymatrix{0 \ar[r] & K \ar[r]^u \ar@{=}[d] & A \ar[r]^g \ar[d] & E \ar[r] \ar[d]^{g'} & 0 \\
0 \ar[r] & K \ar[r] & A' \ar[r] & B \ar[r] & 0}.$$
This gives rise to a short exact sequence
$$0 \to A \to A' \oplus E \to B \to 0.$$
Since $A,B \in \mathcal{A}$ and $\mathcal{A}$ is thick, we get $E \in \mathcal{A}$.  Hence, $K,C \in \mathcal{A}$.
\end{proof}

Thanks to $\mathcal{H}$ being hereditary, the bounded derived category of $\mathcal{H}$, written as $D^b(\mathcal{H})$, is easy to describe. Recall that a stalk complex in $D^b(\mathcal{H})$ is a complex concentrated in one degree, that is, a complex of the form $X[i]$ for an object $X$ in $\mathcal{H}$. From \cite{Len}, every object in $D^b(\mathcal{H})$ is a finite direct sum of stalk complexes. Observe also that given two objects $X,Y \in \mathcal{H}$, the condition $\Hom_{D^b(\mathcal{H})}(X,Y[i]) \ne 0$ implies either $i=0$ or $i=1$.
A full subcategory $\mathcal{A}$ of $D^b(\mathcal{H})$ is \emph{thick} if it closed under direct summands and
whenever we have a map $U\to V$ in $\mathcal{A}$,
then the distinguished triangle
$$W \to U \to V \to W[1]$$
lies in $\mathcal{A}$.  In particular, a thick subcategory is closed under shifts and is a triangulated subcategory of $D^b(\mathcal{H})$. Given a thick subcategory $\mathcal{A}$ of $D^b(\mathcal{H})$, we denote by
$H^0(\mathcal{A})$ the category $\mathcal{H} \cap \mathcal{A}$, that is, the complexes $C$ in $\mathcal{A}$ for which $H^i(C) =0$
for all $i \ne 0$.  The category $H^0(\mathcal{A})$ is abelian and extension-closed, and hence is a thick subcategory of $\mathcal{H}$.  Observe also that it is the heart of the $t$-structure on $\mathcal{A}$ coming from the canonical $t$-structure on $D^b(\mathcal{H})$; see \cite{Ha}. It is easily seen that $\mathcal{A}$ is triangle-equivalent to the bounded derived category of $H^0(\mathcal{A})$.

\medskip

Given a family of objects $E$ in $D^b(\mathcal{H})$, we write $\mathcal{D}(E)$ for the thick subcategory of $D^b(\mathcal{H})$ generated by the objects in $E$. We define $\mathcal{D}(E)^\perp$ (resp.~$\p \mathcal{D}(E)$) to be the full subcategory of $D^b(\mathcal{H})$ generated by the objects $Y$ with $\Hom(X,Y[i])=0$ (resp. $\Hom(Y,X[i])=0$) for all $X \in E$ and all $i \in \mathbb{Z}$. Observe that if $E \subseteq \mathcal{H}$, $\mathcal{C}(E) = H^0(\mathcal{D}(E))$ and $H^0(\mathcal{D}(E)^\perp)=\mathcal{C}(E)^\perp$.
As on the level of abelian categories, an indecomposable object $X$ in $D^b(\mathcal{H})$ is \emph{exceptional} if $\Hom(X,X[i])=0$ for all nonzero integers $i$ and ${\rm End}(X)\cong k$.  An \emph{exceptional sequence} $E=(X_1,X_2,\ldots,X_r)$ in $D^b(\mathcal{H})$ is a sequence of exceptional objects for which $$\Hom(X_i,X_j[l])=0 \; \text{for all} \; l \; \text{and} \; 1 \le i<j \le r.$$
Any exceptional sequence in $\mathcal{H}$ is also an exceptional sequence in $D^b(\mathcal{H})$.  Conversely, if $E=(X_1,X_2,\ldots,X_r)$ is exceptional in $D^b(\mathcal{H})$, then there exist integers $i_1,i_2,\ldots,i_r$ for which $H^0(E) := (X_1[i_1],\ldots,X_r[i_r])$ is an exceptional sequence in $\mathcal{H}$.

\medskip

Let us introduce more terminology.  Let $\mathcal{E}$ be any $k$-linear category.  A full subcategory $\mathcal{F}$ of $\mathcal{E}$ is said to be \emph{contravariantly finite} if for any object $E$ in $\mathcal{E}$, there exists a morphism $f: F \to E$ with $F \in \mathcal{F}$ such that Hom$(F',f)$ is surjective for any $F' \in \mathcal{F}$. Such a morphism $f$ is called a \emph{right $\mathcal{F}$-approximation} of $E$. The following result is stated at the derived category level but is also true at the abelian category level.

\begin{Lemma} \label{lemmaContr}Suppose that $\mathcal{H}$ is Hom-finite. Let $E=(X_1,X_2,\ldots,X_r)$ be an exceptional sequence in $D^b(\mathcal{H})$. Then $\mathcal{D}(E)$ is contravariantly finite in $D^b(\mathcal{H})$.
\end{Lemma}
\begin{proof} We proceed by induction on $r$.  If $r=1$, then $\mathcal{D}(E)$ is the additive category containing the shifts of $X_1$.  This is clearly contravariantly finite in $D^b(\mathcal{H})$.  Suppose now that $r>1$. We have that $E' = (X_1,\ldots,X_{r-1})$ is an exceptional sequence with $\mathcal{D}(E')$ contravariantly finite.  Let $M$ be any object in $D^b(\mathcal{H})$ with $f: C \to M$ a right $\mathcal{D}(E')$-approximation of $M$.  Since $\mathcal{H}$ is abelian and Hom-finite, $\mathcal{H}$, and hence $D^b(\mathcal{H})$, is Krull-Schmidt.  Therefore, we may assume that $f$ is right-minimal, in the sense that if $\varphi: C \to C$ is such that $f = f \varphi$, then $\varphi$ is an isomorphism. Now, we have a triangle
$$\xymatrix{ N \ar[r]^g & C \ar[r]^f & M \ar[r] & N[1]}.$$
We claim that $N$ lies in $\mathcal{D}(E')^\perp$.  Let $u: B \to N$ be a morphism with $B \in \mathcal{D}(E')$. By the octahedral axiom, we have a commutative diagram
$$\xymatrix{ B \ar@{=}[r] \ar[d]^u & B \ar[d]^{gu} & & \\ N \ar[r]^g \ar[d] & C \ar[r]^{f} \ar[d]^w & M \ar[r] \ar@{=}[d] & N[1]\ar[d] \\ F \ar[r] & B' \ar[r]^v & M \ar[r] & F[1]}$$
where $B'$ is an object in $\mathcal{D}(E')$ and all rows and columns are distinguished triangles. Since $f$ is a right $\mathcal{D}(E')$-approximation of $M$, $v$ factors through $f$ and hence, there exists $w': B' \to C$ with $f = fw'w$.  Hence, since $f$ is right minimal, we get that $w$ is a section, which means that $gu=0$.  Hence, $u$ factors through $M[-1]$.  But since $B$ is in $\mathcal{D}(E')$ and $f[-1]$ is a right $\mathcal{D}(E')$-approximation of $M[-1]$, $u$ factors through $f[-1]$ and therefore, $u=0$.  This proves that $N$ lies in $\mathcal{D}(E')^\perp$.

Now, let $\mathcal{X}$ be the thick subcategory of $D^b(\mathcal{H})$ generated by $X_r$, which is contravariantly finite in $D^b(\mathcal{H})$ using the base case $r=1$.  Let $h: X \to N$ be a minimal right $\mathcal{X}$-approximation of $N$.  The octahedral axiom gives
a commutative diagram
$$\xymatrix{ X \ar@{=}[r] \ar[d] & X \ar[d] & & \\ N \ar[r] \ar[d] & C \ar[r]^{f} \ar[d] & M \ar[r] \ar@{=}[d] & N[1]\ar[d] \\ N' \ar[r] & F \ar[r]^w & M \ar[r] & N'[1]}$$
where $N'$ is an object in $\mathcal{D}(E')^\perp \cap \mathcal{X}^\perp = \mathcal{D}(E)^\perp$ and all rows and columns are distinguished triangles. It is easily seen that the morphism $w: F \to M$ is a right $\mathcal{D}(E)$-approximation of $M$.
\end{proof}

\begin{Prop} \label{exceptionalDerived} Let $\mathcal{H}$ be a Hom-finite hereditary abelian $k$-category with a non-full exceptional sequence $E=(X_1,\ldots,X_r)$ in $D^b(\mathcal{H})$ and such that all indecomposable objects in $\mathcal{D}(E)^\perp$ are exceptional. Then $E$ can be extended to an exceptional sequence $E=(X_1,\ldots,X_r,X_{r+1})$.
\end{Prop}

\begin{proof} We only need to prove that
the category $\mathcal{D}(E)^\perp$ is nonzero.
By the assumption, there is an indecomposable object $M$ which is not in $\mathcal{D}(E)$.  By Lemma \ref{lemmaContr}, we have a triangle
$$\xymatrix{ N \ar[r]^g & C \ar[r]^f & M \ar[r] & N[1]},$$
where $f$ is a right $\mathcal{D}(E)$-approximation of $M$.  Then $N$ is nonzero in $\mathcal{D}(E)^\perp$ and has an exceptional object $X_{r+1}$ as a direct summand. This yields an exceptional sequence $E=(X_1,\ldots,X_r,X_{r+1})$.
\end{proof}

In the rest of this section, we specialize to $\mathcal{H} = \rep(Q)$ for an acyclic quiver $Q$ with $n$ vertices. Given two thick subcategories $\mathcal{C}_1,\mathcal{C}_2$ of $\rep(Q)$, we call the pair $(\mathcal{C}_1,\mathcal{C}_2)$ a \emph{semi-orthogonal pair} if $\mathcal{C}_2 = \mathcal{C}_1^\perp$. In order to deal with such pair, we need the following easy observation.

\begin{Lemma} \label{leftrightperp}
Let $\C$ be a finitely generated abelian, extension-closed subcategory of $\rep(Q)$.  Then $\p (\C^\perp) = \C$.
\end{Lemma}

\begin{proof}
The category $\C$ is generated by an exceptional sequence $(X_1,X_2,\ldots,X_r)$ that can be completed to a full exceptional sequence $(X_1,\ldots,X_r,\ldots,X_n)$.  Then $\C^\perp$ is generated by $X_{r+1},\ldots,X_n$. It is clear that $\C \subseteq \bigcap_{i=r+1}^{n}\, \p X_i$ and that $\bigcap_{i=r+1}^{n}\, \p X_i$ is equivalent to the category of representations of a quiver with $n-(n-r) = r$ vertices. If equality does not hold, then $(X_1,\ldots,X_r)$ is not full in $\bigcap_{i=r+1}^{n}\, \p X_i$, which is a contradiction.
\end{proof}

Given a full exceptional sequence $E=(X_1,X_2,\ldots,X_n)$ in $\rep(Q)$ and an integer $0 \le s \le n$, there is a semi-orthogonal pair $$\mathcal{P}(E,s):=(\mathcal{C}(E_{\le s}),\mathcal{C}(E_{>s})),$$ where $E_{\le s} = (X_1,\ldots,X_s)$ and $E_{>s} = (X_{s+1},\ldots,X_n)$. We set $E_{\le 0} = \emptyset$ and $E_{>n}=\emptyset$. It is clear that $\mathcal{C}(E_{>s}) \subseteq \mathcal{C}(E_{\le s})^\perp$. Moreover,
$E_{>s}$ is an exceptional sequence in $\mathcal{C}(E_{\le s})^\perp$, which is equivalent to $\rep(Q')$ for some acyclic quiver $Q'$ with $n-s$ vertices. Since any exceptional sequence in $\rep(Q')$ can be completed to a full exceptional sequence, see \cite{CB}, and $E$ is full, we see that $\mathcal{C}(E_{>s}) = \mathcal{C}(E_{\le s})^\perp$. An exceptional pair of the form $\mathcal{P}(E,s)$ where $E=(X_1,X_2,\ldots,X_n)$ is a full exceptional sequence and $1 \le s \le n$ will be called an \emph{exceptional semi-orthogonal pair}.

\begin{Lemma} \label{repr-finite}Let $Q$ be any acyclic quiver and let $\mathcal{C}$ be any thick subcategory of $\rep(Q)$ which is representation-finite.  Then $\mathcal{C}$ is equivalent to $\rep(Q')$ for some (possibly disconnected) Dynkin quiver $Q'$.
\end{Lemma}

\begin{proof} We prove the statement by induction on the number $r$ of non-isomorphic indecomposable objects in $\mathcal{C}$. If $r \le 1$, the result is clear.  So assume $r \ge 2$. Let $X_0$ be an indecomposable representation in $\mathcal{C}$ of least dimension over $k$. Since $\mathcal{C}$ is thick, we get ${\rm End}(X_0) \cong k$. Using the same argument as in the proof of the lemma in \cite[p. 166]{Ha}, we get $\Ext^1(X_0,X_0)=0$. Hence, $X_0$ is exceptional.  Now, in $\mathcal{C}$, $X_0^\perp$ is thick and representation-finite with fewer non-isomorphic indecomposable objects. By induction, $X_0^\perp$ is equivalent to $\rep(Q'')$ where $Q''$ is a (possibly disconnected) Dynkin quiver. Now, from Proposition \ref{equiv_thick}, $X_0^\perp$ is given by an exceptional sequence $(X_1, \ldots, X_m)$. Let $E$ be the exceptional sequence $(X_0, X_1, \ldots, X_m)$ in $\mathcal{C}$. It is clear that $\mathcal{C}$ is a hom-finite hereditary abelian $k$-category, and $E$ is also an exceptional sequence in the bounded derived category $D^b(\mathcal{C})$ of $\mathcal{C}$. By Proposition \ref{exceptionalDerived}, since $\mathcal{D}(E)^\perp$ is zero, we get that $E$ is full.  By Proposition \ref{equiv_thick}, $\mathcal{C}$ is equivalent to $\rep(Q')$ for some acyclic quiver $Q'$.  Since $\mathcal{C}$ is representation-finite, $Q'$ is a union of quivers of Dynkin type.
\end{proof}

The preceding lemma yields the following: for an acyclic quiver $Q$, any thick subcategory of $\rep(Q)$ which is representation-finite is generated by an exceptional sequence.

\medskip

A \emph{connecting component} of $D^b(\rep(Q))$ is a connected component of the Auslander-Reiten quiver of $D^b(\rep(Q))$ containing projective representations. It is unique if and only if $Q$ is connected. The following result strengthens \cite[Prop. 3.1.10]{Dichev}.

\begin{Prop} Any thick subcategory generated by a finite set of objects in connecting components of $D^b(\rep(Q))$ is generated by an exceptional sequence $E$ where all the terms can be chosen to be in connecting components.
\end{Prop}

\begin{proof} It is sufficient to consider the case where $Q$ is connected and non-Dynkin. In \cite{Dichev}, it is proven that if $\mathcal{C}$ is a thick subcategory of $\rep(Q)$ which is generated by a finite set of preprojective representations, then $\C$ is generated by an exceptional sequence where each term can be chosen to be preprojective. Let $\mathcal{A}$ be a thick subcategory of $D^b(\rep(Q))$ generated by objects $X_1,\ldots,X_r$ where all the $X_i$ lie in the connecting component and hence are exceptional. Let $\tau$ denote the Auslander-Reiten translate in $D^b(\rep(Q))$.  The connecting component contains only the preprojective representations and the inverse shifts of the preinjective representations. There exists a positive integer $t$ for which all $\tau^{-t}X_i$ are preprojective indecomposable representations.  Hence $\mathcal{C}(\tau^{-t}X_1,\ldots,\tau^{-t}X_r)$ is a thick subcategory of $\rep(Q)$ generated by preprojective representations.  Therefore, by \cite[Prop. 3.1.10]{Dichev},
$$\mathcal{C}(\tau^{-t}X_1,\ldots,\tau^{-t}X_r) = \mathcal{C}(E)$$
for an exceptional sequence $E=(Y_1, \ldots,Y_m)$ where all $Y_i$ are preprojective. Clearly, we also have
$$\mathcal{D}(\tau^{-t}X_1,\ldots,\tau^{-t}X_r) = \mathcal{D}(E)$$
where here, $E$ is seen as an exceptional sequence in $D^b(\rep(Q))$. From this, we see that
$$\mathcal{A}=\mathcal{D}(X_1,\ldots,X_r) = \mathcal{D}(E'),$$
where $E'$ is the exceptional sequence $E'=(\tau^tY_1,\ldots,\tau^tY_m)$.
\end{proof}

Using the fact that the thick subcategories $\mathcal{D}$ of $D^b(\rep(Q))$ correspond to the thick subcategories $H^0(\mathcal{D})$ of $\rep(Q)$, we get the following result.

\begin{Cor} \label{Cor2}
Let $Q$ be a connected acyclic quiver. Any thick subcategory of $\rep(Q)$ generated by non-regular representations is generated by an exceptional sequence whose terms can be chosen to be non-regular.
\end{Cor}

For the rest of this section, we specialize to the Euclidean case. Let $Q$ be a Euclidean quiver and let $E=(X_1,X_2,\ldots,X_n)$ be a full exceptional sequence in $\rep(Q)$.  Let $s$  be an integer with $1\le s \le n$.
Clearly, one of the $X_i$ is not regular since $E$ is full. Then, by Lemma \ref{Lemma1}, one of $\mathcal{C}(E_{\le s})$, $\mathcal{C}(E_{>s})$ only contains indecomposable representations that are exceptional.  Since each of $\mathcal{C}(E_{\le s})$, $\mathcal{C}(E_{>s})$ is equivalent to the category of representations of some acyclic quiver, we get that one of $\mathcal{C}(E_{\le s})$, $\mathcal{C}(E_{>s})$ is representation-finite. The following theorem also appears in \cite[Theorem 3.2.15]{Dichev}.

\begin{Theo}(Dichev) \label{thick}
Let $Q$ be a Euclidean quiver. Any thick subcategory of $\rep(Q)$ is either of the form $\mathcal{C}(E)$ for an exceptional sequence $E$ in $\rep(Q)$ or is entirely contained in $\mathcal{R}$eg.
\end{Theo}

\begin{proof} Let $\mathcal{A}$ be a thick subcategory of $\rep(Q)$ which contains at least one preprojective or preinjective indecomposable object $X$.  By Proposition \ref{DieterVossiek}, $\mathcal{A}$ is a Hom-finite hereditary abelian category containing $X$.  Clearly, $X$ is an exceptional representation, hence providing an exceptional sequence $E' = (X)$ in $\mathcal{A}$. From Lemma \ref{Lemma1}, $\mathcal{D}(E')^\perp$ in $D^b(\mathcal{A})$ only contains exceptional objects. By Lemma \ref{exceptionalDerived}, we see that $E'$ can be completed to a full exceptional sequence $E$ in $D^b(\mathcal{A})$, proving that $\mathcal{A}$ is also generated by an exceptional sequence.
\end{proof}

The following lemma is well known; see for example \cite{IT}. We will need
it shortly.

\begin{Lemma} \label{lemmaDynkin}
Let $R$ be a Dynkin quiver.  Any thick subcategory of $\rep(R)$ is equivalent to
$V^\perp$ for a rigid representation $V \in \rep(R)$.  Moreover,
$V^\perp$ is equivalent to $\rep(R')$, where $R'$ is a (possibly disconnected) Dynkin quiver.
\end{Lemma}

In order to understand thick subcategories of $\rep(Q)$ which are contained
in $\mathcal R$eg, we need to classify the thick subcategories of a single
tube.  Let $\mathcal T$ be a tube of rank $r$, which is identified with the additive subcategory of $\rep(Q)$ that it generates.  Let $J$ be a subset of the
quasi-simples of $\mathcal T$.  Write $\mathcal E_J$ for $X^\perp \cap \mathcal{T}$, where
$X$ is the direct sum of the quasi-simples not in $J$.

We say that $(\mathcal E_J,\mathcal F)$ is a \emph{regular orthogonal pair} if
$\mathcal F \subseteq \p \mathcal  E_J \cap \mathcal E_J^\perp$,
$\mathcal F$ is thick in $\mathcal{T}$ and contains only exceptional indecomposables.  In this case, let
$\mathcal S_{J,\mathcal F}$ be the additive hull of $\mathcal E_J$ and
$\mathcal F$, which is clearly a thick subcategory of $\mathcal{T}$. The following result describes all the thick subcategories in a given stable tube. It extends \cite[Prop. 3.2.8]{Dichev}, where connected thick subcategories are considered.

\begin{Prop} \label{thicktube}
Let $\mathcal T$ be a tube in $\rep(Q)$.  Any thick subcategory
of $\mathcal T$ can be written as $\mathcal S_{J,\mathcal F}$ for a unique
subset $J$ of the quasi-simples and subcategory $\mathcal F$ such that
$(\mathcal E_J,\mathcal F)$ is a
regular orthogonal pair.
\end{Prop}

\begin{proof}
Let $\mathcal C$ be a thick subcategory of $\mathcal T$.  Let $J$ be
the set of quasi-socles of the singular-isotropic representations in $\mathcal C$, if any.
We claim that $\mathcal E_J$
is contained in $\mathcal C$, and that if we set $\mathcal F$ to be
the additive hull of $\ind \mathcal C \setminus \ind \mathcal E_J$, then
$(\mathcal E_J,\mathcal F)$ is a regular orthogonal pair, so
$\mathcal C= \mathcal S_{J,\mathcal F}$.

First, we establish that $\mathcal E_J$ is contained in $\mathcal C$.
If $J=\emptyset$, then $\mathcal E_J=0$, so this is obvious.  Assume
otherwise.
Let
$X$ be the direct sum of the quasi-simples not in $J$.  Then
$\mathcal E_J = \mathcal T \cap
X^\perp$.  Suppose that $\mathcal{T}$ has rank $r_1$. Since the summands of $X$ form an exceptional sequence, from Proposition \ref{GL2}, $X^\perp$ is equivalent to the category of representations of a quiver $Q'$ with $|Q_0| - r_1 + |J|$ vertices. It is clear that $Q'$ is a possibly disconnected Euclidean quiver.  Since $\mathcal{T}\cap X^\perp$ contains exactly $|J|$ singular-isotropic representations, we see that the ranks of the tubes of $\rep(Q')$ will be the same as the ones for $\rep(Q)$, but one rank will decrease by $r_1 - |J|$.  If the $r_i$ denote the ranks of the non-homogeneous tubes for $\rep(Q)$, the well know formula $\sum (r_i-1) = n-2$ gives $ (|J|-1) + \sum_{i \ne 1 } (r_i-1) = (n - r_1 + |J|) - 2$ which then tells us that $Q'$ is connected. It follows that $\mathcal E_J$ is equivalent to
some tube $\mathcal T'$ of $\rep(Q')$.  Since the singular-isotropic representations in
$\mathcal T'$ generate all of $\mathcal T'$ as a thick subcategory, it follows
that the smallest thick subcategory containing the singular-isotropic
representations in $\mathcal C$ is $\mathcal E_J$.  Thus, $\mathcal E_J$ is
contained in $\mathcal C$.

Suppose that $J\ne \emptyset$.  Let $\mathcal Q$ denote the set
of objects of $\mathcal E_J$ which correspond to the quasi-simples of
$\mathcal T'$. The objects in $\mathcal E_J$ consist of representations which
have filtrations by objects from $\mathcal Q$.

Now think of the filtration by quasi-simples of the objects from
$\mathcal Q$.  Each $Q$ in $\mathcal Q$ has a filtration $\mathcal K_Q$ by a consecutive
sequence of the quasi-simples; these consecutive sequences are disjoint
and their union is the set of all the quasi-simples of $\mathcal T$.
Suppose we have an indecomposable object $X\in \mathcal F$.  Consider its
filtration by quasi-simples, which also gives rise to a
consecutive sequence of quasi-simples.  Since $X \not \in \mathcal E_J$,
this sequence of quasi-simples is not the concatenation of subsequences
corresponding to elements of $\mathcal Q$: it either begins, or ends, or
both, out of step with the subdivision of quasi-simples of $\mathcal T$
into the
sets $\mathcal K_Q$.  We would like to show that the quasi-simples in the
filtration of $X$ all lie inside $\mathcal K_Q$ for some $Q$, and do not
include either the quasi-socle or the quasi-top of that $Q$.  Suppose that
this is not the case.  Then there is some $Q \in \mathcal Q$ such that $X$ admits a
non-epimorphism to $Q$, or a non-monomorphism from $Q$.  Suppose we are in
the first case.  (The second is dual.)  Since $\mathcal C$ is thick,
by Proposition \ref{DieterVossiek}
$\mathcal C$ contains the image $A$ of $X$ in $Q$, so that there
is a
short exact sequence $0\rightarrow A \rightarrow Z \rightarrow B \rightarrow
0$ in $\C$, where $Z$ is singular-isotropic and lies in $\mathcal E_J$. But then, there is another short exact sequence
$0\rightarrow B \rightarrow Z' \rightarrow A \rightarrow
0$ where $Z'$ is singular-isotropic.  Since $\C$ is extension-closed, $Z' \in \C$ and since $Z'$ is singular-isotropic, we must have $Z' \in \mathcal E_J$. This means that there exists $Q' \in \mathcal Q$ which lies on the co-ray where $Z', X$, and $A$ lie.  This contradicts the fact that $\mathcal K_Q$ and $\mathcal K_{Q'}$ have to be disjoint.
Therefore, we know that, for any indecomposable
$X$ in $\mathcal F$, there is some
$Q\in \mathcal Q$ such that $X$ admits a filtration by the quasi-simples
in the filtration of $Q$, excluding its quasi-socle and quasi-top.  This
implies, in particular, that
$X\subseteq \p \mathcal E_J \cap \mathcal E_J^\perp$.  It also shows that
$X$ is necessarily exceptional.

Now we consider the case that $J=\emptyset$.  The situation which we must
rule out is that $\mathcal C$ contains some non-rigid indecomposables,
but no singular-isotropic representations.  That this is impossible is
a by-product of the proof of Lemma \ref{intersections}.
\end{proof}

Now we are able to describe the semi-stable subcategories of $\rep (Q)$ which
lie in $\mathcal R$eg.

\begin{Prop} \label{semistablenonempty} A thick subcategory of $\mathcal R$eg is semi-stable if and
only if it
contains all the homogeneous tubes and its intersection with each
non-homogeneous tube $\mathcal T_i$
is of the form $\mathcal S_{J_i,\mathcal F_i}$ where each $J_i$ is non-empty.
\end{Prop}

\begin{proof}
By Theorem \ref{one}, a semi-stable subcategory of $\mathcal R$eg can also be
written as $\mathcal A \cap \mathcal R$eg for $\mathcal A$ some finitely
generated extension-closed abelian subcategory, equivalent to the representations of a
(possibly disconnected) Euclidean quiver.  Since $\mathcal A$ is finitely
generated, it can be written as $V^\perp$ for some rigid object $V$, and
since $\mathcal A$ is equivalent to a possibly disconnected Euclidean
quiver, $V$ must be regular.

Since $V$ is regular, $V^\perp$ contains all the homogeneous tubes.
Now consider $V_i$, the maximal direct summand of $V$ lying in $\mathcal T_i$.
Since $V_i$ is rigid, $V_i$ is contained in some wing, and thus $V_i^\perp$
contains some singular-isotropic representation.  It follows that
$V^\perp \cap \mathcal T_i$ is of the form $\mathcal S_{J_i,\mathcal F_i}$ where
$J_i$ is non-empty.

Conversely, suppose we have a thick subcategory as described in the
statement of the proposition.  We want to show that it is of the form
$V^\perp \cap \mathcal R$eg where $V$ is rigid.  Clearly, it suffices to consider the
case of one tube $\mathcal T$, and a subcatgory $\mathcal S_{J,\mathcal F}$,
with $J \ne \emptyset$.  We want to show that there is some rigid representation
$V \in \mathcal T$ such that $V^\perp \cap \mathcal{T} = \mathcal S_{J,\mathcal F}$.

Let $\mathcal Q$ be the collection of the (relative) quasi-simples of $\mathcal E_J$, as
before.  For each $A\in \mathcal Q$, let $R_A$ be the target of the
irreducible epimorphism from $A$ (or $R_A=0$ if there is no such epimorphism).  For $A,B \in \mathcal Q$, if $\Ext(R_A,R_B) \ne 0$, then there is a nonzero morphism $g: R_B \to \tau R_A$. Since there is an epimorphism $f: B \to R_B$ and a monomorphism $h: \tau R_A \to A$, this yields a nonzero radical morphism $hgf: B \to A$, contradicting that $A,B$ are quasi-simples. Therefore, $\bigoplus_{A \in \mathcal{Q}} R_A$ is rigid. For $A \in \mathcal{Q}$, denote by $\mathcal{W}_A$ the extension-closed abelian subcategory generated by all the quasi-simple composition factors of $A$ except its quasi-top and quasi-socle.  We see that
$(\bigoplus_{A \in \mathcal Q} R_A)^\perp$ consists of the additive hull of $\mathcal E_J$
together with each of the wings $\mathcal{W}_A$.  Since each of the $\mathcal{W}_A$ is representation-finite
of type $\mathbb A$, Lemma \ref{lemmaDynkin} tells us that
if we add further summands to $\bigoplus R_A$,
it is possible to find a
rigid object $V$ such that $V^\perp \cap \mathcal{T} = \mathcal S_{J,{\mathcal F}}$.
\end{proof}

Using the classification of thick subcategories inside the regular representations, the  previous proposition can be restated as follows:

\begin{Prop}\label{oneone}
The semi-stable subcategories in (ii) of Theorem \ref{one}
can also be described as those abelian, extension-closed
subcategories of the regular part of $\rep(Q)$, which
contain infinitely many indecomposable objects from each tube.
\end{Prop}

We end this section with the following, which will be used in the next section.

\begin{Prop} \label{except}
Let $Q$ be a Euclidean quiver and $\mathcal{C}$ be a thick subcategory of $\rep(Q)$. Then one of $\mathcal{C}$, $\mathcal{C}^\perp$ is not contained in $\mathcal{R}$eg if and only if both $\mathcal{C}$ and $\mathcal{C}^\perp$ are generated by exceptional sequences.  In this case, $(\mathcal{C}, \mathcal{C}^\perp)$ is an exceptional semi-orthogonal pair.
\end{Prop}

\begin{proof}
If $\C$ is not contained in $\mathcal{R}$eg, then it follows from Theorem \ref{thick} that $\C$ is generated by an exceptional sequence.  Since this exceptional sequence can be completed to a full exceptional sequence, see \cite{CB}, $\C^\perp$ is also generated by an exceptional sequence.  Similarly, if $\C^\perp$ is not contained in $\mathcal{R}$eg, then $\C^\perp$ is generated by an exceptional sequence by Theorem \ref{thick}. By Lemma \ref{leftrightperp}, we have $\p(\C^\perp) = \C$ and hence $\C$ is generated by an exceptional sequence.
Suppose now that both $\C$, $\C^\perp$ are contained in $\mathcal{R}$eg. If $\C$ is generated by an exceptional sequence, then all the representations in homogeneous tubes are contained in $\C^\perp$ and hence, $\C^\perp$ cannot be generated by an exceptional sequence.  If $\C^\perp$ is generated by an exceptional sequence, then all the representations in homogeneous tubes are contained in $\p (\C^\perp) = \C$ and $\C$ cannot be generated by an exceptional sequence.
\end{proof}

\section{Intersection of semi-stable subcategories}

In this section, we start with $Q$ any connected acyclic quiver and later specialize to the Euclidean case. We first look at some situations where the intersection of semi-stable subcategories is again semi-stable, and we end the section by describing, in the Euclidean case, how to construct the whole set of subcategories of $\rep(Q)$ arising as an intersection of semi-stable subcategories. As already seen, we may assume that our semi-stable subcategories are of the form $\rep(Q)_d$ where $d$ is a dimension vector.

A prehomogeneous dimension vector $d$ such that all the indecomposable summands of $M(d)$ are non-regular will be called \emph{strongly prehomogeneous}. The following result says that when $d_1,d_2$ are strongly prehomogeneous, the intersection $\rep(Q)_{d_1} \cap \rep(Q)_{d_2}$ remains semi-stable.

\medskip

\begin{Prop} \label{stronglyprehomo}Let $d_1,d_2$ be two strongly prehomogeneous dimension vectors.  Then $\rep(Q)_{d_1} \cap \rep(Q)_{d_2} = \rep(Q)_{d_3}$, where $d_3$ is  prehomogeneous.
\end{Prop}

\begin{proof}
By assumption, $\rep(Q)_{d_i} = M({d_i})^\perp$, $i=1,2$, where the $M({d_i})$ are rigid representations whose indecomposable direct summands are non-regular.  Now, $\rep(Q)_{d_1} \cap \rep(Q)_{d_2} = \C(M({d_1}),M({d_2}))^\perp$.  However, from Corollary \ref{Cor2}, $$\C(M({d_1}),M({d_2})) = \C(E),$$ where $E$ is an exceptional sequence whose terms are non-regular.  This gives a rigid representation $V$ with $\C(M({d_1}),M({d_2}))=\C(V)$ by Proposition \ref{equiv_thick}.  Then, $\rep(Q)_{d_1} \cap \rep(Q)_{d_2} = V^\perp$.  Since $V$ is rigid, the orbit of $V$ is open in $\rep(Q,d_V)$.  Hence, the canonical decomposition of $V$ is given by the dimension vectors of its indecomposable direct summands. In particular, $d_3$ is prehomogeneous.
\end{proof}

Here is a simple examples that illustrate some results of Section \ref{sectionthick} and Proposition \ref{stronglyprehomo}.

\begin{Exam}  Let $Q$ be the following quiver
$$\xymatrixrowsep{12pt}\xymatrixcolsep{12pt}\xymatrix{ & 2 \ar[dl] & \\ 1  & & 3 \ar[ll] \ar[ul]}$$
There is only one non-homogeneous tube in $\rep(Q)$.  It is of rank two with two quasi-simple representations $M_1,M_2$ where $d_{M_1} = (1,0,1)$ and $d_{M_2}=(0,1,0)$.  Observe that $\C(M_1,M_2)$ is not generated by an exceptional sequence.  This shows that
the assumption that $M_1,M_2$ are non-regular in Corollary \ref{Cor2} is essential.

For $i=1,2,3$, denote by $S_i$ the simple representation at $i$, by $P_i$ the projective representation at $i$ and by $I_i$ the injective representation at $i$. Observe that $d_{S_1}, d_{S_3}$ are strongly prehomogeneous.  Consider
$\rep(Q)_{d_{S_1}} \cap \rep(Q)_{d_{S_3}}={\rm add}(I_2)$.  By Proposition \ref{stronglyprehomo}, we can write $\rep(Q)_{d_{S_1}} \cap \rep(Q)_{d_{S_3}}$ as
$\rep(Q)_{d_3}$ for some $d_3$ which is prehomogeneous but not necessarily
strongly prehomogeneous.  We will now determine $d_3$.

Observe that $\rep(Q)_{d_3}={\rm add}(I_2) = V^\perp$ for some $V$. Since $I_2$ is supported at vertices $2,3$, we have that $\langle d_V, - \rangle |_{Q'} = \langle d_3, - \rangle |_{Q'}$ where $Q'$ is the full subquiver of $Q$ generated by vertices $2,3$; see Corollary \ref{cor1}. This gives $d_V |_{Q'} = d_3 |_{Q'}$. The condition $\langle d_V, d_{I_2} \rangle =0$ gives $d_V = (a,0,b)$ for some $a,b \ge 0$. Hence, $d_3 = (c,0,b)$ for some integer $c$. This gives that $I_2$ lies in $\rep(Q)_{(c,0,b)}$ for any $c,b$ with $b \ge 0$. However, for some values of $b,c$, $\rep(Q)_{(c,0,b)}$ properly contains add$(I_2)$. If $c=b$, then by Proposition \ref{King}, $S_1$ lies in $\rep(Q)_{(c,0,b)}$, which is impossible. Similarly, if $c=0$, then $M_1$ lies in $\rep(Q)_{(c,0,b)}$, which is impossible. Finally, if $b=0$, then $S_2$ lies in $\rep(Q)_{(c,0,b)}$, which is again impossible.
Hence, $b,c$ are non-zero and are not equal. An easy check gives that $d_3=(c,0,b)$ is cp-equivalent to one of the vectors $(1,0,2),(2,0,1), (-1,0,1)$ whose canonical presentations are $(1,0,1) + (0,0,1), (1,0,1) +(1,0,0), (-1,0,0) + (0,0,1)$, respectively.

\medskip

Let $d_1 = (1,0,0)$ and $d_2 = (2,1,1)$.   Since $d_1$ is the dimension vector of $P_1$, we have $$\rep(Q)_{d_1} = P_1^\perp = {\rm add}(S_2,S_3,I_2) = \C(I_3,I_2).$$  Similarly, $$\rep(Q)_{d_2} = P_3^\perp = {\rm add}(S_1,S_2,P_2) = \C(P_1, P_2).$$ Therefore, each of $\rep(Q)_{d_1}, \rep(Q)_{d_2}$ is generated by a representation whose indecomposable direct summands are non-regular. We see that the intersection $\rep(Q)_{d_1} \cap \rep(Q)_{d_2}$ is $\C(P_1,P_3)^\perp = {\rm add}(S_2)$, where $S_2$ is regular.
\end{Exam}

From now on, we suppose that $Q$ is a Euclidean quiver. We give a complete description of the possible intersections of semi-stable subcategories of $\rep(Q)$, and in particular, a description of those intersections that are not semi-stable. Let us start with some notations and reminders.

\medskip

Given two semi-orthogonal pairs $(\mathcal{A},\mathcal{A}^\perp)$ and $(\mathcal{B},\mathcal{B}^\perp)$, we define
$$(\mathcal{A},\mathcal{A}^\perp)\ast(\mathcal{B},\mathcal{B}^\perp) = (\C(\mathcal{A},\mathcal{B}),\mathcal{A}^\perp \cap \mathcal{B}^\perp),$$
where we recall that $\C(\mathcal{A},\mathcal{B})$ denotes the smallest thick subcategory of $\rep(Q)$ containing $\mathcal{A}$ and $\mathcal{B}$. Observe that $\mathcal{A}^\perp \cap \mathcal{B}^\perp = \C(\mathcal{A},\mathcal{B})^\perp$, hence defining a new semi-orthogonal pair $(\C(\mathcal{A},\mathcal{B}),\mathcal{A}^\perp \cap \mathcal{B}^\perp)$.

Recall that for each facet $F_I$ of $H_\delta^{ss}$, we have the representation-finite thick subcategory $\mathcal{W}_I$ generated by the exceptional representations whose dimension vectors lie in $F_I$. We also have the simplicial cone $C_{I}$ which is generated by $F_{I}$ and
$\delta$.
Given a dimension vector $d$, recall that $M(\hat d)$ is a rigid representation of dimension vector $\hat d$, where $\hat d$ is the sum of the real Schur roots appearing in the canonical decomposition of $d$. We start with the following.

\begin{Prop} \label{prop5.7}
Let $Q$ be a Euclidean quiver and $d_1,d_2$ be dimension vectors.  Then $(\mathcal{C}(M(\hat d_1)),\mathcal{C}(M(\hat d_1))^\perp)\ast(\mathcal{C}(M(\hat d_2)),\mathcal{C}(M(\hat d_2))^\perp)$ is exceptional if and only if one of the following occurs:
\begin{enumerate}[$(a)$]
    \item At least one of $d_1,d_2$ does not lie on $H_{\delta}^{ss}$,
    \item There exists $I \in R$ with $d_1,d_2 \in C_I$.
\end{enumerate}
\end{Prop}

\begin{proof}
If at least one of $d_1,d_2$ is not in $H_\delta^{ss}$, then $\C(M(\hat d_1),M(\hat d_2))$ is not contained in $\mathcal{R}$eg. We apply Proposition \ref{except} in this case.
If both $d_1,d_2$ lie in the same $C_I$, then $\C(M(\hat d_1),M(\hat d_2))$ is clearly representation-finite by Lemma \ref{Lemma3}. By Lemma \ref{repr-finite}, $\C(M(\hat d_1),M(\hat d_2))$ is generated by an exceptional sequence and so is $\C(M(\hat d_1),M(\hat d_2))^\perp$. Suppose now that $d_1,d_2$ both lie in $H_{\delta}^{ss}$ but in different cones.  If $\C(M(\hat d_1),M(\hat d_2))$ is representation-finite, then it lies in some $\mathcal{W}_I$ by Lemma \ref{Lemma3}, and hence both $d_1,d_2$ lie in $C_I$, a contradiction. Since $\C(M(\hat d_1),M(\hat d_2))^\perp$ contains all the representations in homogeneous tubes, it is not representation-finite.  Therefore, $(\mathcal{C}(M(\hat d_1)),\mathcal{C}(M(\hat d_1))^\perp)\ast(\mathcal{C}(M(\hat d_2)),\mathcal{C}(M(\hat d_2))^\perp)$ cannot be exceptional since neither $\C(\mathcal{C}(M(\hat d_1)),\mathcal{C}(M(\hat d_2)))$ nor $\C(\mathcal{C}(M(\hat d_1)),\mathcal{C}(M(\hat d_2)))^\perp$ is representation-finite.
\end{proof}

From what we just proved, if $(\mathcal{C}(M(\hat d_1)),\mathcal{C}(M(\hat d_1))^\perp)\ast(\mathcal{C}(M(\hat d_2)),\mathcal{C}(M(\hat d_2))^\perp)$ is not exceptional, then both $\C(M(\hat d_1),M(\hat d_2))$ and $\mathcal{C}(M(\hat d_1))^\perp \cap \mathcal{C}(M(\hat d_2))^\perp$ are contained in $\mathcal{R}$eg and are not representation-finite. The following proposition gives a first partial answer on how to compute the intersection of two semi-stable subcategories in the Euclidean case.

\begin{Prop} Let $Q$ be a Euclidean quiver and $d_1,d_2$ be dimension vectors satisfying the equivalent conditions of Proposition \ref{prop5.7}.
Then
$\rep(Q)_{d_1}\cap\rep(Q)_{d_2} = \rep(Q)_{d_3}$ for some dimension vector $d_3$.
\begin{enumerate}[$(a)$]
    \item If at least one of $d_1,d_2$ is not in $H_\delta^{ss}$, then $d_3$ is not in $H_\delta^{ss}$.  In particular, $\rep(Q)_{d_3}$ is representation-finite.
\item If both $d_1,d_2$ lie on some facet $F_I$, then $d_3$ can be chosen to be in $F_I$ and $\rep(Q)_{d_3} = V^\perp$ where $V$ is rigid and $\C(V)$ is representation-finite.
\item If both $d_1,d_2$ lie in some $C_I$ but not both on $F_I$, then $d_3$ can be chosen to be in $C_I$. In this case, $\rep(Q)_{d_3} = V^\perp \cap \mathcal{R}$eg where $V$ is rigid and $\C(V)$ is representation-finite.
\end{enumerate}
\end{Prop}

\begin{proof}
The main statement follows from Propositions \ref{equiv_thick}, \ref{dreg} and \ref{prop5.7}. For $(a)$, suppose that $d_1$ does not lie in $H_\delta^{ss}$.  Then $\rep(Q)_{d_1}$ is representation-finite and so is $\rep(Q)_{d_1} \cap \rep(Q)_{d_2} = \rep(Q)_{d_3}$.  If $d_3$ lies in $H_\delta^{ss}$, then all the homogeneous tubes are contained in $\rep(Q)_{d_3}$, contradicting that $\rep(Q)_{d_3}$ is representation-finite. For $(b)$, $d_1 = \hat d_1$ and $d_2 = \hat d_2$ and we have $M(d_1),M(d_2) \in \mathcal{W}_I$, and hence, $\C(M(d_1),M(d_2)) \subseteq \mathcal{W}_I$ is representation-finite.  Therefore, $\C(M(d_1),M(d_2))$ is generated by some rigid representation $V$ and $\rep(Q)_{d_3} = V^\perp$.  For part $(c)$, observe that, for $i=1,2$, we have $\rep(Q)_{d_i} = \rep(Q)_{\hat d_i} \cap \mathcal{R}{\rm eg}$ where $\hat d_i \in F_I$. Hence, from part $(b)$, the intersection $\rep(Q)_{d_1}\cap\rep(Q)_{d_2}$ is given by $\rep(Q)_{d_3'} \cap \mathcal{R}{\rm eg}$ where $d_3' \in F_I$ and $\rep(Q)_{d_3'} = V^\perp$ for some rigid $V$ such that $\C(V)$ is representation-finite. The result follows by setting $d_3 = d_3' + \delta$.
\end{proof}

Now, we concentrate on the remaining case, that is, to describe the intersection $\rep(Q)_{d_1} \cap \rep(Q)_{d_2}$ when both ${d_1}$ and ${d_2}$ lie in $H_\delta^{ss}$, and not in the same $C_I$. We can clearly assume that $Q$ has more than $2$ vertices.

\begin{Lemma} If $d_1$ and $d_2$ both lie in $H_\delta^{ss}$, and do not both
lie in the same $C_I$, then $\rep(Q)_{d_1} \cap \rep(Q)_{d_2}$ is
contained in $\mathcal R$eg.
\end{Lemma}

\begin{proof}
We can clearly assume, from Proposition \ref{dreg}, that both $d_1,d_2$ lie on the boundary of $H_\delta^{ss}$. Let $\mathcal C=\rep(Q)_{d_1}\cap \rep(Q)_{d_2}$.
Let $f=d_1+d_2$.  We observe that $X=M(d_1)\oplus M(d_2)$ is a representation
with dimension vector $f$, such that $\Hom(X,Y)=0$ and $\Ext^1(X,Y)=0$
for any $Y$ in $\mathcal C$.  Thus $\mathcal C$ is contained in
$\rep(Q)_{f}$.  Since $d_1$ and $d_2$ do not both lie in the same
$C_I$, we know that $f$ lies in the interior of $H_\delta^{ss}$.
Proposition \ref{dreg} tells us that $\rep(Q)_f$ is contained in
$\mathcal R$eg, proving the result.
\end{proof}

Since we now know that, in the case we are presently studying,
$\rep(Q)_{d_1}\cap \rep(Q)_{d_2}$ is contained in $\mathcal R$eg, it suffices
to restrict our attention to the regular parts of $\rep(Q)_{d_1}$ and
$\rep(Q)_{d_2}$.  It therefore suffices to assume that neither $d_i$ lies
on the boundary of $H_\delta^{ss}$.  If $d_1$ does lie on the boundary,
replace it by $d_1'=d_1+\delta$, and observe that
$\rep(Q)_{d'_1}= \mathcal R\mathrm {eg} \cap \rep(Q)_{d_1}$.

By Proposition \ref{semistablenonempty}, we know exactly what kind of subcategories can
arise as semi-stable subcategories of $\mathcal R$eg: namely, they are
the subcategories with the property that their intersection with each
tube is of the form $\mathcal S_{J,\mathcal F}$ where $J$ is non-empty. Observe that the condition $J \ne \emptyset$ applied to a homogeneous tube $\mathcal{T}$ just means $\mathcal S_{J,\mathcal F} = \mathcal{T}$ (and $\mathcal{F} = 0$).

Given a non-homogeneous tube $\mathcal{T}$ of $\rep(Q)$, denote by $\mathcal{O}_\mathcal{T}$ the set of indecomposable representations in $\mathcal{T}$ which admit an irreducible monomorphism to a singular-isotropic representation. The set $\mathcal{O}_\mathcal{T}$ hence form a $\tau$-orbit of $\mathcal{T}$, which lies just below the $\tau$-orbit of singular-isotropic representations.

\begin{Lemma} \label{leminter}
Let $\mathcal T$ be a non-homogeneous tube, and $\mathcal S_{J_i,\mathcal F_i}$ be
thick subcategories of $\mathcal T$ with each $J_i$ non-empty.
Let $\mathcal C=\bigcap_i \mathcal S_{J_i,\mathcal F_i}$.  Then
\begin{enumerate}[$(a)$]
    \item The category $\mathcal C$ is given by some $\mathcal S_{K,\mathcal G}$,
where $\mathcal G$ does not contain any indecomposables from $\mathcal O_\mathcal T$.
    \item Conversely, any thick subcategory $\mathcal S_{K,\mathcal G}$ of $\mathcal T$
    such that $\mathcal G$ does not contain any indecomposables from
$\mathcal O_\mathcal T$
can be written as an intersection $\mathcal S_{K_1,\mathcal{G}_1} \cap \mathcal S_{K_2,\mathcal{G}_2}$ where $K_1, K_2$ are non-empty.
\end{enumerate}
\end{Lemma}

\begin{proof} We begin by proving $(a)$.
Clearly, $K=\bigcap_i J_i$.  If $K\ne \emptyset$,
let $V$ be a singular-isotropic
representation contained in $\mathcal C$.  Since $\mathcal G \subseteq
\p V \cap V^\perp$, it follows that $\mathcal G\cap \mathcal O_\mathcal T = \emptyset$,
as desired.

Suppose now that $K=\emptyset$.  Since $\bigcap_i \mathcal E_{J_i} =
0$, any indecomposable of $\bigcap_i \mathcal S_{J_i,\mathcal F_i}$ must be
contained in some $\mathcal F_i$.  But by the previous argument, $\mathcal
F_i\cap \mathcal O_\mathcal T=\emptyset$.

Now we prove the converse direction.
We may assume that $K=\emptyset$, otherwise, we may take all $\mathcal S_{K_i,\mathcal {G}_i}$ equal $\mathcal S_{K,\mathcal G}$. Let $Z$ be a singular-isotropic representation in $\mathcal T$ with quasi-socle $S$. We claim that for an exceptional object $M$ in $\mathcal T$, we have $\Hom(M,Z) \ne 0$ if and only if $M$ has $S$ as a quasi-simple composition factor. Fix $M$ exceptional in $\mathcal T$. Suppose that $M$ has $S$ as a quasi-simple composition factor. If $M$ has $S$ as quasi-socle, then we have a monomorphism $M \to Z$ and $\Hom(M,Z) \ne 0$. Otherwise $M$ has a quotient $N$ that has $S$ as quasi-socle. The projection $M \to N$ followed by the monomorphism $N \to Z$ gives a nonzero morphism $M \to Z$ and hence, $\Hom(M,Z) \ne 0$. Conversely, assume that $M$ does not have $S$ as a quasi-simple composition factor. Let $f: M \to Z$ be a morphism with image $C$. Since $S$ is the quasi-socle of $Z$, either $C=0$ or $S$ is the quasi-socle of $C$. The latter case gives that $S$ is a quasi-simple composition factor of $M$, a contradiction. Therefore, $C=0$ and hence $f=0$, which gives $\Hom(M,Z)=0$. This proves the claim. Similarly, one can prove that for an exceptional object $M$ in $\mathcal T$, we have $\Ext^1(M,Z) \ne 0$ if and only if $M$ has $S$ as a quasi-simple composition factor.

For an exceptional object $L$ in $\mathcal{O}_T$, denote by $\mathcal{W}_L$ the extension-closed abelian subcategory of $\mathcal T$ generated by the quasi-simple composition factors of $L$. Now from Lemma \ref{Lemma3}, we see that $\mathcal{G}$ is contained in $\mathcal{W}_L$ for some $L \in \mathcal{O}_T$. Observe that there is at least one quasi-simple object $S$ in $\mathcal{W}_L$ such that $S$ is not a quasi-simple composition factor of any object in $\mathcal{G}$. Otherwise, since $\mathcal{G}$ is thick, we would get $L \in \mathcal{G}$, a contradiction. Let $M_1$ be the indecomposable object of maximal quasi-length in $\mathcal{W}_L$ that has $S$ as quasi-socle, and take $N_1 = \tau^{-1}M_1$. Moreover, let $N_2$ be the indecomposable object of maximal quasi-length in $\mathcal{W}_L$ that has $S$ as quasi-top. Observe that $\mathcal{G} \subseteq (N_1 \oplus N_2)^\perp \subseteq \mathcal{W}_L$. Now, $(N_1 \oplus N_2)^\perp = \mathcal{C}_1 \coprod \mathcal{C}_2$ where each $\mathcal{C}_i$, provided it is non-zero, is equivalent to the category of representations of a Dynkin quiver. From Lemma \ref{lemmaDynkin}, there exists a rigid $C_i \in \mathcal{C}_i$ such that $\mathcal{G} \cap \mathcal{C}_i = C_i^\perp \cap \mathcal{C}_i$. Thus, we have $\mathcal{G} = (N_1 \oplus N_2 \oplus C_1 \oplus C_2)^\perp$.
By Proposition \ref{thicktube}, each $(N_i \oplus C_i)^\perp$ can be expressed as some $\mathcal S_{K_i,\mathcal{G}_i}$, so $\mathcal G= \mathcal S_{K_1,\mathcal G_1}\cap \mathcal S_{K_2,\mathcal G_2}$.
Now, $S$ is not a quasi-simple composition factor of $N_1 \oplus C_1$, so
$K_1$ contains $S$ and if $U$ is the unique quasi-simple object not in $\mathcal{W}_L$, then $U$ is not a quasi-simple composition factor of $N_2 \oplus C_2$, so $U\in K_2$.  It follows that the expression we have obtained for
$\mathcal G$ as the intersection of $\mathcal S_{K_1,\mathcal G_1}$ and
$\mathcal S_{K_2,\mathcal G_2}$ is of the desired form.
\end{proof}

The following example illustrates the second part of the Lemma and the idea of its proof.

\begin{Exam}
In the picture below, we have a tube of rank $6$ where only the exceptional indecomposable objects are drawn.
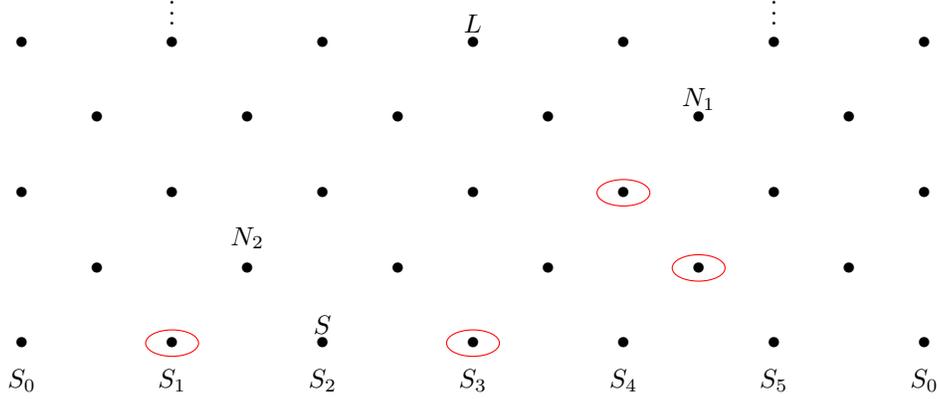
\begin{figure}[h]
  \centering
  \begin{tikzpicture}[xscale=1,yscale=0.5]

\node at (-4, 6) {$\vdots$};
\node at (4, 6) {$\vdots$};

\node at (-6,-3) {$\bullet$};
\node at (-4,-3) {$\bullet$};
\node at (-2,-3) {$\bullet$};
\node at (-0,-3) {$\bullet$};
\node at (2,-3) {$\bullet$};
\node at (4,-3) {$\bullet$};
\node at (6,-3) {$\bullet$};

\node at (-6,-4) {$S_0$};
\node at (-4,-4) {$S_1$};
\node at (-2,-4) {$S_2$};
\node at (-0,-4) {$S_3$};
\node at (2,-4) {$S_4$};
\node at (4,-4) {$S_5$};
\node at (6,-4) {$S_0$};

\node at (-6,1) {$\bullet$};
\node at (-4,1) {$\bullet$};
\node at (-2,1) {$\bullet$};
\node at (-0,1) {$\bullet$};
\node at (2,1) {$\bullet$};
\node at (4,1) {$\bullet$};
\node at (6,1) {$\bullet$};

\node at (-5,-1) {$\bullet$};
\node at (-3,-1) {$\bullet$};
\node at (-1,-1) {$\bullet$};
\node at (1,-1) {$\bullet$};
\node at (3,-1) {$\bullet$};
\node at (5,-1) {$\bullet$};

\node at (-5,3) {$\bullet$};
\node at (-3,3) {$\bullet$};
\node at (-1,3) {$\bullet$};
\node at (1,3) {$\bullet$};
\node at (3,3) {$\bullet$};
\node at (5,3) {$\bullet$};

\node at (-6,5) {$\bullet$};
\node at (-4,5) {$\bullet$};
\node at (-2,5) {$\bullet$};
\node at (-0,5) {$\bullet$};
\node at (2,5) {$\bullet$};
\node at (4,5) {$\bullet$};
\node at (6,5) {$\bullet$};

\node at (-2,-2.5) {$S$};
\node at (0,5.5) {$L$};

\node at (-3,-0.2) {$N_2$};
\node at (3,3.5) {$N_1$};

\draw (-4,-3) [red] circle (10pt);
\draw (3,-1) [red] circle (10pt);
\draw (0,-3) [red] circle (10pt);
\draw (2,1) [red] circle (10pt);

  \end{tikzpicture}
  \caption{A tube of rank $6$}
\label{fig:A3}
\end{figure}

We have a thick subcategory $\mathcal{G}$ whose indecomposable objects are circled. Using the notations of the proof of Lemma \ref{leminter}, take $C_1 = S_5$ and $C_2 = 0$. Then each $(N_i \oplus C_i)^\perp$ is given by $S_{K_i,\mathcal{G}_i}$, where $K_i$ is non-empty.
\end{Exam}

We have now essentially proved our final main theorem.

\begin{Theo}\label{inters} Let $Q$ be a Euclidean quiver.
There are finitely many  subcategories of $\rep(Q)$
which arise as an intersection of semi-stable
subcategories, and which are not themselves semi-stable subcategories
for any stability condition. Moreover, such a subcategory is characterized by the following.
\begin{enumerate}[$(a)$]
    \item It is contained entirely
in the regular part of $Q$.
\item It contains all
the homogeneous tubes.
\item Its intersection with each non-homogeneous tube
$\mathcal T_i$ can be written as $\mathcal S_{J_i,\mathcal F_i}$ where $(\mathcal{E}_{J_i}, \mathcal F_i)$ is a regular orthogonal pair in $\mathcal{T}_i$ with $\mathcal F_i\cap \mathcal O_{\mathcal{T}_i} = \emptyset$.  Moreover, at least one of the $J_i$ is empty.
\item It can be
expressed as the intersection of at most two semi-stable subcategories.
\end{enumerate}
\end{Theo}

\begin{proof} We have shown that if $\mathcal C_1$ and $\mathcal C_2$ are
two semi-stable subcategories, such that $\mathcal C_1\cap \mathcal C_2$
is not itself semi-stable, then $\mathcal C_1\cap \mathcal C_2$ is
contained in the regular part $\mathcal{R}$eg, and contains the homogeneous
tubes.  Since it is thick, its intersection with each non-homogeneous
tube can be written as $\mathcal S_{J_i,\mathcal F_i}$, and since, by assumption,
it is not semi-stable, some $J_i$ must be empty.  Lemma \ref{leminter}
says further that $\mathcal F_i \cap \mathcal O_{\mathcal{T}_i} =\emptyset$.  The
same lemma also shows that any such subcategory can be written as the
intersection of at most two semi-stable subcategories.
\end{proof}

\bigskip

{\sc Acknowledgments:} The authors are supported by NSERC while the second author is also supported in part by AARMS.  The third author would like to thank
MSRI for its hospitality during the completion of this paper.  The authors
would like to express their gratitude to an anonymous referee for a careful
reading of the paper, which led to several improvements.

\end{document}